\documentclass[12pt,a4paper]{amsart}

\usepackage[utf8]{inputenc}
\usepackage{enumerate,a4,titletoc,graphicx}
\usepackage{bbm,amssymb,mathtools,mathrsfs,xspace}
\usepackage[usenames,dvipsnames]{xcolor}
\usepackage[pagebackref,colorlinks,linkcolor=BrickRed,citecolor=OliveGreen,urlcolor=black,hypertexnames=true]{hyperref}
\usepackage{tikz}
\usepackage[margin=2.5cm]{geometry}

\DeclareFontFamily{U}{mathx}{\hyphenchar\font45}
\DeclareFontShape{U}{mathx}{m}{n}{
	<5> <6> <7> <8> <9> <10>
	<10.95> <12> <14.4> <17.28> <20.74> <24.88>
	mathx10
}{}
\DeclareSymbolFont{mathx}{U}{mathx}{m}{n}
\DeclareFontSubstitution{U}{mathx}{m}{n}
\DeclareMathAccent{\widecheck}{0}{mathx}{"71}

\DeclareMathOperator{\Nilp}{Nilp}
\DeclareMathOperator{\id}{id}
\DeclareMathOperator{\Hom}{Hom}
\DeclareMathOperator{\nc}{nc}

\DeclareMathOperator{\LAC}{IC}
\DeclareMathOperator{\cb}{cb}
\DeclareMathOperator{\rk}{rk}

\DeclareMathOperator{\End}{End}

\DeclareMathOperator{\GL}{GL}
\DeclareMathOperator{\tr}{tr}

\DeclareMathOperator{\opm}{M}

\newcommand{\ve}{\varepsilon}
\newcommand{\N}{\mathbb{N}}
\newcommand{\R}{\mathbb{R}}
\newcommand{\C}{\mathbb{C}}
\newcommand{\B}{\mathbb{B}}

\newcommand{\BB}{\mathbb{B}}

\newcommand{\cA}{\mathcal{A}}
\newcommand{\cB}{\mathcal{B}}
\newcommand{\cC}{\mathcal{C}}

\newcommand{\cI}{\mathcal{I}}

\newcommand{\cM}{\mathcal{M}}

\newcommand{\cO}{\mathcal{O}}
\newcommand{\cS}{\mathcal{S}}

\newcommand{\cT}{\mathcal{T}}
\newcommand{\cU}{\mathcal{U}}

\newcommand{\mm}{\mathbbm m}

\newcommand{\kk}{\mathbbm k}
\newcommand{\ckk}{\overline{\kk}}

\newcommand{\lac}{canonical intertwining conditions\xspace}
\newcommand{\Lac}{Canonical intertwining conditions\xspace}

\newcommand{\gX}{\Xi}

\newcommand{\glc}{\GL_n(\C)}
\newcommand{\uc}{\operatorname{U}_n(\C)}
\newcommand*{\mat}[1]{\opm_{#1}(\kk)}
\newcommand*{\matc}[1]{\opm_{#1}(\C)}
\newcommand{\mats}{\opm_s(\kk)}
\newcommand{\matn}{\opm_n(\kk)}
\newcommand{\matcs}{\opm_s(\C)}
\newcommand{\matcn}{\opm_n(\C)}
\newcommand*{\gl}[1]{\GL_{#1}(\kk)}
\newcommand{\gm}{\operatorname{GM}_n}
\newcommand{\tn}{\operatorname{T}_n}
\newcommand{\ud}{\operatorname{UD}_n}
\newcommand{\gms}{\widehat{\operatorname{GM}}_n}
\newcommand{\uds}{\widehat{\operatorname{UD}}_n}
\newcommand{\matser}{\opm_n(\kk[[\xi]])}

\newcommand{\matquot}{\opm_n(\kk((\xi)))}

\newcommand{\Oa}{\cO^{\operatorname{a}}}
\newcommand{\Oua}{\cO^{\operatorname{ua}}}
\newcommand{\Ma}{\cM^{\operatorname{a}}}
\newcommand{\Mua}{\cM^{\operatorname{ua}}}

\newcommand{\Langle}{\mathop{<}\!}
\newcommand{\Rangle}{\!\mathop{>}}
\newcommand{\mx}{\Langle x \Rangle}

\newcommand{\px}{\kk\!\mx}
\newcommand{\py}{\kk\!\Langle y \Rangle}
\newcommand{\pyc}{\C\!\Langle y \Rangle}
\newcommand{\ser}{\kk\!\Langle\!\!\mx\!\!\Rangle}
\newcommand{\dser}[1]{{#1}\!\Langle\!\!\mx\!\!\Rangle}
\newcommand{\usf}{\cU}

\makeatletter
\def\moverlay{\mathpalette\mov@rlay}
\def\mov@rlay#1#2{\leavevmode\vtop{
		\baselineskip\z@skip \lineskiplimit-\maxdimen
		\ialign{\hfil$#1##$\hfil\cr#2\crcr}}}
\makeatother

\newcommand{\pxc}{\C\!\mx}
\newcommand{\serc}{\C\!\Langle\!\!\mx\!\!\Rangle}
\newcommand{\matserc}{\opm_n(\C[[\xi]])}
\newcommand{\matquotc}{\opm_n(\C((\xi)))}

\newtheorem{thm}{Theorem}[section]
\newtheorem{lem}[thm]{Lemma}
\newtheorem{cor}[thm]{Corollary}
\newtheorem{prop}[thm]{Proposition}
\newtheorem{thmA}{Theorem}

\theoremstyle{definition}
\newtheorem{defn}[thm]{Definition}
\newtheorem{exa}[thm]{Example}

\theoremstyle{remark}
\newtheorem{rem}[thm]{Remark}

\numberwithin{equation}{section}

\begin{document}
	
\setcounter{tocdepth}{3}
\contentsmargin{2.55em} 
\dottedcontents{section}[3.8em]{}{2.3em}{.4pc} 
\dottedcontents{subsection}[6.1em]{}{3.2em}{.4pc}
\dottedcontents{subsubsection}[8.4em]{}{4.1em}{.4pc}

\makeatletter
\def\enddoc@text{\ifx\@empty\@translators \else\@settranslators\fi
	\ifx\@empty\addresses \else\@setaddresses\fi
	\newpage\tableofcontents\newpage\printindex}
\makeatother

\setcounter{page}{1}

\title[Local theory of free nc functions]{Local theory of free noncommutative functions: germs, meromorphic functions and Hermite interpolation}

\author[I. Klep]{Igor Klep${}^1$}
\address{Igor Klep, Department of Mathematics, University of Ljubljana}
\email{igor.klep@fmf.uni-lj.si}
\thanks{${}^1$Supported by the Marsden Fund Council of the Royal Society of New Zealand. 
	Partially supported by the Slovenian Research Agency grants J1-8132, N1-0057 and P1-0222.}

\author[V. Vinnikov]{Victor Vinnikov${}^2$}
\address{Victor Vinnikov, Department of Mathematics, Ben-Gurion University of the Negev}
\email{vinnikov@math.bgu.ac.il}
\thanks{${}^2$Supported by the Deutsche Forschungsgemeinschaft (DFG) Grant No. SCHW 1723/1-1}

\author[J. Vol\v{c}i\v{c}]{Jurij Vol\v{c}i\v{c}${}^3$}
\address{Jurij Vol\v{c}i\v{c}, Department of Mathematics, Texas A\&M University}
\email{volcic@math.tamu.edu}
\thanks{${}^3$Supported by the Deutsche Forschungsgemeinschaft (DFG) Grant No. SCHW 1723/1-1}

\date{\today}


\begin{abstract}
Free analysis is a quantization of the usual function theory much like operator space theory is a quantization of classical functional analysis. Basic objects of free analysis are noncommutative functions. These are maps on tuples of matrices of all sizes that preserve direct sums and similarities.

This paper investigates the local theory of noncommutative functions. The first main result shows that for a {\it scalar} point $Y$, the ring $\Oua_Y$ of uniformly analytic noncommutative germs about $Y$ is an integral domain and admits a universal skew field of fractions, whose elements are called meromorphic germs. A corollary is a local-global rank principle that connects ranks of matrix evaluations of a matrix $A$ over $\Oua_Y$ with the factorization of $A$ over $\Oua_Y$. Different phenomena occur for a semisimple tuple of {\it non-scalar} matrices $Y$. Here it is shown that $\Oua_Y$ contains copies of the matrix algebra generated by $Y$. In particular, there exist nonzero nilpotent uniformly analytic functions defined in a neighborhood of $Y$, and $\Oua_Y$ does not embed into a skew field. 
Nevertheless, the ring $\Oua_Y$ is described as the completion of a free algebra with respect to the vanishing ideal at $Y$. This is a consequence of the second main result, a free Hermite interpolation theorem: if $f$ is a noncommutative function, then for any finite set of semisimple points and a natural number $L$ there exists a noncommutative polynomial that agrees with $f$ at the chosen points up to differentials of order $L$. All the obtained results also have analogs for (non-uniformly) analytic germs and formal germs.
\end{abstract}

\if
Free analysis is a quantization of the usual function theory much like operator space theory is a quantization of classical functional analysis. Basic objects of free analysis are noncommutative functions. These are maps on tuples of matrices of all sizes that preserve direct sums and similarities.

This paper investigates the local theory of noncommutative functions. The first main result shows that for a scalar point $Y$, the ring $O_Y$ of uniformly analytic noncommutative germs about $Y$ is an integral domain and admits a universal skew field of fractions, whose elements are called meromorphic germs. A corollary is a local-global rank principle that connects ranks of matrix evaluations of a matrix $A$ over $O_Y$ with the factorization of $A$ over $O_Y$. Different phenomena occur for a semisimple tuple of non-scalar matrices $Y$. Here it is shown that $O_Y$ contains copies of the matrix algebra generated by $Y$. In particular, there exist nonzero nilpotent uniformly analytic functions defined in a neighborhood of $Y$, and $O_Y$ does not embed into a skew field. 
Nevertheless, the ring $O_Y$ is described as the completion of a free algebra with respect to the vanishing ideal at $Y$. This is a consequence of the second main result, a free Hermite interpolation theorem: if $f$ is a noncommutative function, then for any finite set of semisimple points and a natural number $L$ there exists a noncommutative polynomial that agrees with $f$ at the chosen points up to differentials of order $L$. All the obtained results also have analogs for (non-uniformly) analytic germs and formal germs.
\fi


\maketitle


\section{Introduction}

The study of analytic functions in noncommuting variables goes back to the seminal work of Taylor \cite{Tay1,Tay2}. Recently, noncommutative function theory or {\it free analysis} saw a rapid development fueled by free probability, dilation theory, operator systems and spaces, control theory and optimization \cite{BGM,Pop2,Voi,MS,HKM,KVV3,AM1}. The central objects are {\it noncommutative} (nc) functions $f$ defined on tuples of square matrices of finite size that respect basis change and direct sums (see Subsection \ref{ss:nc} for a precise definition). For example,
$$f(x_1,x_2)=x_1\exp\left(x_2(x_1x_2-x_2x_1)^{-1}\right)$$
is an nc function defined on all pairs of matrices $(X_1,X_2)$ such that $X_1X_2-X_2X_1$ is nonsingular. Thus $f$ is not defined for any pair of scalar matrices, but it is defined on an open set in $\matcn^2$ for $n>1$. The key idea of functions respecting basis change and direct sums (albeit in the setting of $C^*$-algebras) was first addressed by Takesaki \cite{Tak}.

Nc functions admit a differential calculus and posses extraordinary analytic properties. If an nc function $f$ is bounded on a neighborhood of $Y$, then $f$ is continuous and even analytic there, and equals its noncommutative power series expansion about $Y$ determined by its differentials at $Y$ \cite{Voi,HKM1,KVV3}. The precise nature of convergence depends on the underlying topology on tuples of matrices of all sizes. When using the disjoint union topology, we obtain {\it analytic} nc functions. Another natural option is the uniformly open topology generated by noncommutative balls about matrix points (see Subsection \ref{ss:ua}), in which case we talk about {\it uniformly analytic} nc functions. While the methods for dealing with analytic nc functions derive mainly from complex analysis, uniformly analytic nc functions are closer to operator space theory. In both settings, several classical analytic results have their free analogs, such as the implicit/inverse function theorem \cite{AKV,AM2}, the Oka--Weil approximation theorem \cite{AM0}, the Nevanlinna--Pick interpolation \cite{Pop3}, Choquet theory \cite{DK}, a homogeneous Nullstellensatz \cite{SSS}, the Jacobian conjecture \cite{Pas} and the Grothendieck theorem \cite{Aug}.

This paper addresses the local behavior of analytic nc functions. Given a matrix point $Y$, a {\it (uniformly) analytic noncommutative germ} about $Y$ is the equivalence class of a (uniformly) analytic nc function on a neighborhood of $Y$. By the previous paragraph, such a germ is determined by the power series expansion of an nc function. For this reason we also define {\it formal germs} as formal noncommutative power series about $Y$ satisfying certain natural linear constraints, called {\it \lac} about $Y$ (Definition \ref{d:lac}). Roughly speaking, these conditions encode preservation of similarity and direct sums behavior of nc functions, so that a (uniformly) analytic germ is precisely a formal germ given by a (uniformly) convergent power series satisfying \lac. This paper presents the first systematic study of algebras of noncommutative germs with a view toward functional calculus.

\subsection*{Main results and guide to the paper}

For a matrix point $Y\in\matcs^g$ let $\cO_Y$, $\Oa_Y$ and $\Oua_Y$ denote the $\C$-algebras of formal, analytic, and uniformly analytic germs in $g$ freely noncommuting variables $x=(x_1,\dots,x_g)$, respectively. After preliminary results in Section \ref{sec2}, our study of these algebras branches in two directions, depending on whether $Y$ is a scalar point or not.

If $Y$ is a scalar point ($s=1$), then $\cO_Y$ is isomorphic to $\serc$, the noncommutative power series in $x$. A formal rational expression in elements of $\serc$ is called a {\it meromorphic expression}. One can attempt to evaluate a meromorphic expression $m$ at a $g$-tuple $\gX^{(n)}$ of $n\times n$ generic matrices with independent commuting entries; if all the inverses appearing through the calculation exist, then the output is an $n\times n$ matrix of commutative power series. On the set of meromorphic expressions admitting evaluation at $\gX^{(n)}$ for at least one $n$ we impose the following equivalence relation: $m_1\sim m_2$ if and only if $m_1(\gX^{(n)})=m_2(\gX^{(n)})$ whenever both sides exist. The equivalence classes are called {\it formal meromorphic germs}, and form the universal skew field of fractions of $\serc$ in the sense of Cohn \cite{Coh}; see Theorem \ref{t:formal}. From a purely algebraic perspective, this result places $\serc$ among the sporadic examples of rings with explicit universal skew fields of fractions \cite{KVV2,KVV}. On the other hand, the universality specifies what it means for a meromorphic expression to be identically zero, which is essential for analysis because it allows us to talk about functions induced by meromorphic expressions. More concretely, we prove the Amitsur--Cohn theorem for meromorphic identities. An algebra $\cA$ is {\it stably finite} if for every $n\in\N$ and $A,B\in \cA^{n\times n}$, $AB=I$ implies $BA=I$. For instance, a $C^*$-algebra with a faithful trace, such as a  type II\textsubscript{1} von Neumann algebra, is stably finite, while the full algebra of bounded operators on a separable Hilbert space is not; see \cite{RLL,Bla} for details and more examples.

\begin{thmA}\label{t:a}
Let $m=m(x_1,\dots,x_g)$ be a meromorphic expression. The following are equivalent:
\begin{enumerate}
\item for every $n$, $m(\gX^{(n)})$ is either $0$ or undefined;
\item for every stably finite algebra $\cA$ and $a_1,\dots,a_g\in\cA$, if there is a homomorphism $\serc\to\cA$ given by $x_i\mapsto a_i$, then $m(a_1,\dots,a_g)$ is either $0$ or undefined.
\end{enumerate}
\end{thmA}

See Theorem \ref{t:id} for the proof. If $Y$ is a scalar point, one can similarly start with rational expressions in elements of $\Oa_Y$ or $\Oua_Y$, and compare their evaluations on matrix points close to $Y$ (in the suitable topology). This results in (uniformly) meromorphic germs, $\Ma_Y$ and $\Mua_Y$, which are universal skew fields of fractions of $\Oa_Y$ and $\Oua_Y$, respectively (Corollary \ref{c:univ}). This universality property together with Theorem \ref{t:a} has two important consequences. Firstly, uniformly meromorphic germs can be evaluated in stably finite Banach algebras (Corollary \ref{c:sf}), which is fundamental for the functional calculus of meromorphic nc functions. Secondly, we obtain the following local-global rank principle for (uniformly) analytic functions and ranks of their evaluations on matrix points close to $Y$ (Theorem \ref{t:b}). The {\it inner rank} of a matrix $A$ over a ring $R$ \cite{Sch,Coh} is the smallest $r$ such that $A=BC$  for some matrices $B$ with $r$ columns and $C$ with $r$ rows over $R$.

\begin{thmA}\label{t:b}
Let $Y\in\C^g$. The inner rank of a matrix $A$ over $\Oa_Y$ (resp. $\Oua_Y$) equals
$$\max\left\{
\frac{\rk A(X)}{n}\colon n\in\N,\ X\in\matcn^g\text{ \rm in a neighborhood of $Y$}
\right\}.$$
\end{thmA}

\noindent While Theorem \ref{t:b} does not mention meromorphic germs or the universal property, they are crucial for its proof in Theorem \ref{t:rank} below. 

Our last result pertaining to analytic nc functions about the origin concerns the action of $\GL_n(\C)$ on $\matcn^g$ via simultaneous conjugation. Then the formal meromorphic germs are closely related to meromorphic $\GL_n(\C)$-concomitants (equivariant maps) $q:\matcn^g\to\matcn$. More precisely, for every concomitant $q$ there exist $f_1,f_2\in\serc$ such that $q=f_1(\gX^{(n)})f_2(\gX^{(n)})^{-1}$; see Theorem \ref{t:inv}.

Now suppose $Y\in\matcs^g$ is a non-scalar point. We say that $Y$ {\it semisimple} is if every invariant subspace for $Y$ admits a complementary invariant subspace. In this case let $\cS(Y)$ be the unital $\C$-subalgebra of $\matcs$ generated by $Y_1,\dots,Y_g$. We show that the algebras $\cO_Y$, $\Oa_Y$ and $\Oua_Y$ for a non-scalar semisimple point $Y$ are not domains and thus do not admit skew fields of fractions. Moreover, Corollary \ref{c:hom} below implies the following.

\begin{thmA}\label{t:c}
If $Y$ is a semisimple point, then $\Oua_Y$ contains a subalgebra isomorphic to $\cS(Y)$. Hence if $\cS(Y)=\matc{s}$, then $\Oua_Y\cong \opm_s(\cA)$ for some algebra $\cA$.

In particular, if a semisimple $Y$ is not similar to a direct sum of scalar points, then there exist uniformly analytic nc functions f about $Y$ such that $f\neq0$ and $f^2=0$.
\end{thmA}

The existence of nilpotent analytic nc functions about semisimple points 
poses questions about the structure of germ algebras that are endemic to the non-scalar case. Our main tool for answering them is the following novel Hermite interpolation result for nc functions.

\begin{thmA}\label{t:d}
Let $f$ be an nc function, $S$ a finite collection of semisimple points in its domain, and $L\in\N$. Then there exists a noncommutative polynomial $p$ such that $f$ and $p$ agree on $S$ up to their noncommutative differentials of order $L$.
\end{thmA}

A more general version is given in Theorem \ref{t:hermite}, and the degree of the interpolating polynomial can be  explicitly estimated. In contrast with other interpolation results for nc functions \cite{Pop1,Pop3,BMV}, Theorem \ref{t:d} is the first one to approximate nc functions with polynomials in \emph{non-scalar} points up to higher order differentials. Also, Theorem \ref{t:d} fails without the semisimplicity assumption; see \cite{AM1} for an example where not even a value of an nc function at a non-semisimple point can be attained by a polynomial. 

The first consequence of our interpolation theorem is Corollary \ref{c:compl} which offers a deeper understanding of the formal germs in terms of the free algebra $\pxc$:

\begin{thmA}\label{t:e}
Let $Y$ be a semisimple point and $\cI(Y)=\{p\in \pxc\colon p(Y)=0 \}$. Then	
$$\cO_Y = \varprojlim_\ell \big(\pxc/\cI(Y)^\ell\big).$$
\end{thmA}

Furthermore, we classify the germ algebras about $Y$ up to isomorphism in terms of $Y$ as follows.
\begin{thmA}\label{t:f}
If $Y$ and $Y'$ are semisimple points,  then
$$\cO_{Y}\cong \cO_{Y'}
\quad\iff\quad
\Oa_{Y}\cong \Oa_{Y'}
\quad\iff\quad
\Oua_{Y}\cong \Oua_{Y'}
\quad\iff\quad
\cS(Y)\cong \cS(Y').$$
\end{thmA}
See Theorem \ref{t:iso} for the proof. Finally, in Section \ref{sec7} we describe a uniformly analytic nc function on a neighborhood of $Y$ with finitely many prescribed differentials at $Y$ that is minimal in a certain sense. This construction is quite different from the aforementioned polynomial interpolation. It unveils some additional features of the germ algebras and provides examples of nc functions with unusual properties, such as ones in Theorem \ref{t:c}.

\section{Preliminaries}\label{sec2}

Let $\kk$ be a field of characteristic $0$ and fix $g\in\N$. Let $x=\{x_1,\dots,x_g\}$ be freely noncommuting variables. Let $\mx$ be the free monoid over $x$, $\px$ the free $\kk$-algebra over $x$, and $\ser$ the completion of $\px$ with respect to the $(x_1,\dots,x_g)$-adic topology. This is the topology induced by the norm on $\px$ given as follows: $\|f\|=2^{-d}$ if $f$ belongs to the ideal $(x_1,\dots,x_g)^d$ but not to $(x_1,\dots,x_d)^{d+1}$. The elements of $\px$ and $\ser$ are called {\bf noncommutative (nc) polynomials} and {\bf noncommutative (nc) power series}, respectively.

For $X\in\mat{m}^g$, $Y\in\mat{n}^g$ and $S,T\in\matn$ we write
\begin{align*}
X\oplus Y &= (X_1\oplus Y_1,\dots, X_g\oplus Y_g)\in\mat{m+n}^g, \\
SYT &= (SY_1T,\dots, SY_gT)\in\matn^g, \\
[S,Y] &= (SY_1-Y_1S,\dots,SY_g-Y_gS)\in\matn^g.
\end{align*}
Furthermore, we write $\oplus^n Y$ for $Y\oplus\cdots\oplus Y$ with $n$ summands. Also, $\otimes$ denotes both the tensor product over $\kk$ and the Kronecker product of matrices.

\subsection{Noncommutative functions}\label{ss:nc}

Let us follow the terminology and definitions of \cite{KVV3}. A {\bf noncommutative (nc) space} over $\kk^g$ is
$$\kk^g_{\nc} = \bigsqcup_{n\in\N} \matn^g=\bigsqcup_{n\in\N} (\kk^g)^{n\times n}.$$
In particular, $\kk_{\nc}=\bigsqcup_n \matn$. For $\Omega\subseteq \kk^g_{\nc}$ we write $\Omega_n=\Omega \cap \matn^g$. We say that $\Omega\subseteq \kk^g_{\nc}$ is a {\bf noncommutative (nc) set} if $X\oplus Y\in\Omega$ for every $X,Y\in\Omega$. A map $f:\Omega\to \kk_{\nc}$ on an nc set is a {\bf noncommutative (nc) function} if
\begin{enumerate}
	\item $f$ is graded, $f(\Omega_n)\subseteq \matn$ for all $n$;
	\item $f$ respects direct sums, $f(X\oplus Y)=f(X)\oplus f(Y)$ for all $X,Y\in\Omega$;
	\item $f$ respects similarities, $f(SXS^{-1})=Sf(X)S^{-1}$ for all $X\in\Omega_n$ and $S\in\gl{n}$ such that $SXS^{-1}\in\Omega_n$.
\end{enumerate}
For a more thorough treatment of free analysis and noncommutative function theory see \cite{Voi,KVV3,AM1}.

\subsection{Differential operators}

An nc set $\Omega$ is {\bf (right) admissible} if for every $X\in\Omega_m$, $Y\in\Omega_n$ and $Z\in(\kk^{m\times n})^g$ there exists $\alpha\in\kk\setminus\{0\}$ such that
$$\begin{pmatrix}X & \alpha Z \\ 0 & Y\end{pmatrix}\in\Omega_{m+n}.$$
Let $f$ be an nc function on an admissible set $\Omega$, $Y\in\Omega_s$ and $\ell\in\N$. Then the {\bf $\ell$-th order (right) noncommutative (nc) differential operator at $Y$} is the $\ell$-linear map
$$\Delta_Y^\ell f\colon (\mats^g)^\ell\to \mats$$
determined by
\begin{equation}\label{e:diff}
f\begin{pmatrix}
Y & Z^1 &  & \\
& \ddots & \ddots & \\
& & \ddots & Z^\ell \\
& & & Y
\end{pmatrix}
=\begin{pmatrix}
f(Y) & \Delta_Y^1f(Z^1) & \cdots & \Delta_Y^\ell f(Z^1,\dots,Z^\ell)\\
& \ddots & \ddots & \vdots\\
& & \ddots & \Delta_Y^1f(Z^\ell) \\
& & & f(Y)
\end{pmatrix}.
\end{equation}
The particular block structure in \eqref{e:diff} is due to $f$ being noncommutative; cf. \cite[Theorem 3.11]{KVV3}. For convenience we write $\Delta_Y^0f=f(Y)$.

\subsubsection{Ampliations}

Let $V,W$ be vector spaces over $\kk$ and let $T:V^\ell\to W$ be an $\ell$-linear map. Then $T$ can be viewed as a linear map $T:V^{\otimes\ell}\to W$. For every $n$ we can naturally extend it to a linear map $(V^{\otimes\ell})^{n\times n}\to W^{n\times n}$ by block-wise application of $T$. By composing it with the canonical map
$$(V^{n\times n})^{\otimes\ell} \to (V^{\otimes\ell})^{n\times n}$$
we obtain an $\ell$-linear map $T_n:(V^{n\times n})^{\ell}\to W^{n\times n}$. Whenever $n$ is clear from the context, we simply write $T$ instead of $T_n$.

As a special case, an $\ell$-linear map $T:(\mats^g)^\ell\to \mats$ extends to an $\ell$-linear map $T:(\mat{ns}^g)^\ell\to \mat{ns}$ as above using the identification $\mats^{n\times n}=\mat{ns}$.

\begin{rem}
In \cite{KVV3}, this amplified map is denoted as
$$T_n(Z^1,\dots,Z^\ell)= (Z^1\odot_s \cdots \odot_s Z^\ell) T$$
for $Z^i\in\mat{ns}^g = ((\kk^g)^{s\times s})^{n\times n}$ using the faux product $\odot_s$ for $n\times n$ matrices over the tensor algebra $\mathbb{T}((\kk^g)^{s\times s})$.
\end{rem}

Returning to the differential operators, we have the following relation between amplifying points in the nc space and amplifying operators \cite[Proposition 3.3]{KVV3}:
\begin{equation}\label{e:amp}
(\Delta_{\oplus^n Y}^\ell f) (Z^1,\dots,Z^\ell) = \left(\Delta_Y^\ell f\right)_n (Z^1,\dots,Z^\ell).
\end{equation}

\subsubsection{\Lac}

Since nc functions respect direct sums and similarities, their differential operators satisfy certain intertwining conditions, which we describe next.

\begin{defn}\label{d:lac}
Let $Y\in\mats^g$. A sequence $(f_\ell)_{\ell=0}^\infty$ of $\ell$-linear maps
$$f_\ell:\left(\mats^g\right)^\ell\to \mats$$
satisfies the {\bf \lac} with respect to $Y$ (shortly $\LAC(Y)$) if
$$f_1([S,Y])=[S,f_0]$$
and for $\ell\ge2$,
\begin{align*}
f_\ell([S,Y],Z^1,\dots, Z^{\ell-1})
& =S f_{\ell-1}(Z^1,\dots, Z^{\ell-1})-f_{\ell-1}(SZ^1,Z^2,\dots, Z^{\ell-1}), \\
f_\ell(\dots,Z^j,[S,Y], Z^{j+1},\dots)
& = f_{\ell-1}(\dots,Z^{j-1},Z^jS,Z^{j+1},\dots) \\
 &\quad -f_{\ell-1}(\dots,Z^j,SZ^{j+1},Z^{j+2},\dots), \\
f_\ell(Z^1,\dots Z^{\ell-1},[S,Y])
& =f_{\ell-1}(Z^1,\dots,Z^{\ell-2}, Z^{\ell-1} S)-f_{\ell-1}(Z^1,\dots, Z^{\ell-1})S
\end{align*}
for $1\le j\le \ell-2$ and all $S\in\mats$, $Z^j\in\mats^g$.
\end{defn}

\begin{rem}
If $Y\in\kk^g$, then $\LAC(Y)$ are void.
\end{rem}

\begin{rem}\label{r:diff}
If $f$ is an nc function on an admissible set $\Omega$ and $Y\in\Omega_s$, then
$$ \left(\Delta^\ell_Y f\right)_{\ell=0}^\infty$$
satisfies $\LAC(Y)$ by \cite[Remark 4.3]{KVV3}.
\end{rem}

\subsection{Topologies on a noncommutative space and analyticity of noncommutative functions}

In this subsection let $\kk=\C$. We will consider two natural topologies on $\C^g_{\nc}$. For $X=(X_1,\dots,X_g)\in\matcn^g$ let $\|X\|_n$ denote the maximum of the operator norms of $X_j$; when $n$ is clear from the context, we simply write $\|\cdot\|$. The norms $\|\cdot\|_n$ on $\matcn^g=(\C^g)^{n\times n}$ correspond to $\C^g$ viewed as an operator space via Ruan's theorem \cite{Pau,Pis}. For an $\ell$-linear map $T:(\matcs^g)^\ell\to\mats$ we have
$$\|T_n\|_{ns}=\max\left\{\|T_n(Z^1,\dots,Z^\ell)\|_{ns}\colon \|Z^1\|_{ns}=\cdots=\|Z^\ell\|_{ns}=1\right\}.$$
If
$$\|T\|_{\cb}=\sup_n\|T_n\|_{ns}$$
is finite, then $T$ is {\bf completely bounded} (in the sense of Christensen and Sinclair \cite[Chapter 17]{Pau}; see also \cite[Proposition 7.49]{KVV3}). When $\ell=1$, this definition of course agrees with the usual notion of a completely bounded linear map. Actually, the results of this paper apply to any operator space structure on $\C^g$, but we restrict to this standard one for the sake of simplicity.

\subsubsection{Disjoint union topology}

A subset $\Omega\subseteq \C^g_{\nc}$ is open in the {\bf disjoint union topology} if $\Omega_n$ is open in the Euclidean topology on $\matcn^g$ for all $n\in\N$ (\cite[Section 7.1]{KVV3}; cf. fine topology \cite[Example 3.3]{AM1}). Let $f$ be an nc function on an open nc set $\Omega$. We say that $f$ is {\bf locally bounded} on $\Omega$ if for every $X\in\Omega$, $f$ is bounded on some open neighborhood of $X$. 

Let $Y\in \matcs^g$. If $f$ is a locally bounded nc function on some nc neighborhood of $Y$, then for every $n\in\N$ we have
$$\limsup_{\ell\to\infty}\sqrt[\ell]{
\sup_{\|Z\|_{ns}=1}	\|\Delta_{\oplus^nY}^\ell f(Z,\dots,Z)\|_{ns}
}<\infty$$
and for all $X$ in some Euclidean neighborhood $U_{ns}\subseteq\matc{ns}^g$ of $\oplus^nY$,
\begin{equation}\label{e:old1}
f(X)=\sum_{\ell=0}^\infty \Delta_{\oplus^nY}^\ell f (X-\oplus^nY,\dots,X-\oplus^nY)
\end{equation}
holds and the series \eqref{e:old1} converges absolutely and uniformly on $U_{ns}$ \cite[Theorem 7.8]{KVV3}. On the other hand, if $(f_\ell)_{\ell=0}^\infty$ is a sequence of multilinear maps satisfying $\LAC(Y)$ and
$$\limsup_{\ell\to\infty}\sqrt[\ell]{
\sup_{\|Z\|_{ns}=1}	\|(f_\ell)_n (Z,\dots,Z)\|_{ns}
}<\infty$$
for all $n$, then the series
\begin{equation}\label{e:old2}
\sum_{\ell=0}^\infty f_\ell(X-\oplus^nY,\dots,X-\oplus^nY)
\end{equation}
converges absolutely and uniformly for $X$ in some open neighborhood $U_{ns}\subseteq \matc{ns}^g$ of $Y$, for all $n\in\N$ \cite[Theorem 8.8]{KVV3}. Furthermore, if $\bigsqcup_n U_{ns}\subset\C^g_{\nc}$ contains an nc set $\Omega$ that is open in the disjoint union topology, then $\eqref{e:old2}$ is a locally bounded nc function on $\Omega$.

\subsubsection{Uniformly open topology}\label{ss:ua}

For $Y\in\matcs^g$ and $\ve>0$ let
$$\BB_\ve(Y)=\bigsqcup_{n\in\N}\left\{X\in\matc{ns}^g\colon \left\|X-\oplus^nY\right\|_{ns}<\ve\right\}\subset \C^g_{\nc}$$
be the {\bf noncommutative (nc) ball} about $Y$ of radius $\ve$. By \cite[Proposition 7.12]{KVV3}, nc balls form a basis of a topology on $\C^g_{\nc}$, which is called the {\bf uniformly open topology} (\cite[Section 7.2]{KVV3}; cf. fat topology \cite[Example 3.5]{AM1}). Let $\Omega$ be a uniformly open nc set. An nc function $f$ is {\bf uniformly locally bounded} on $\Omega$ if for every $X\in\Omega$, $f$ is bounded on some nc ball about $X$. Similarly to the disjoint union topology case, uniform local boundedness is related to uniform analyticity.

\begin{thm}[{\cite[Theorems 7.21 and 8.11]{KVV3}}]\label{t:old4}
Let $Y\in \matcs^g$.

If $f$ is a uniformly locally bounded nc function on some uniformly open nc neighborhood of $Y$, then
\begin{equation}\label{e:old3}
f(X)=\sum_{\ell=0}^\infty \Delta_{Y}^\ell f(X-\oplus^nY,\dots,X-\oplus^nY)
\end{equation}
for all $X$ in some nc ball about $Y$, where the series \eqref{e:old3} converges absolutely and uniformly, and $\limsup_{\ell\to\infty}\sqrt[\ell]{\|\Delta_Y^\ell f\|_{\cb}}<\infty$.

Conversely, let $(f_\ell)_{\ell=0}^\infty$ be a sequence of multilinear maps satisfying $\LAC(Y)$ and $$\limsup_{\ell\to\infty}\sqrt[\ell]{\|f_\ell\|_{\cb}}<\infty.$$
Then the series
\begin{equation}\label{e:old4}
\sum_{\ell=0}^\infty f_\ell(X-\oplus^nY,\dots,X-\oplus^nY)
\end{equation}
converges absolutely and uniformly for all $X$ in some nc ball about $Y$, and $\eqref{e:old4}$ is a locally bounded nc function on that nc ball.
\end{thm}

The infinite sums \eqref{e:old1} and \eqref{e:old3} are called {\bf Taylor-Taylor series} about $Y$ \cite[Chapter 4]{KVV3}; see also \cite{Tay2,Voi} for earlier accounts.

\subsection{Noncommutative germs}\label{ssec2.4}

To study the local behavior of (uniformly) analytic nc functions we define the following.

\begin{defn}\label{d:germ}
Let $Y\in\matcs^g$. An {\bf analytic nc germ about $Y$} is an equivalence class of analytic nc functions about $Y$, where two nc functions are equivalent if they agree on a disjoint union open nc neighborhood of $Y$. Analogously we define a {\bf uniformly analytic nc germ about $Y$} (using the uniformly open topology). The $\C$-algebras of analytic and uniformly analytic nc germs about $Y$ are denoted $\Oa_Y$ and $\Oua_Y$, respectively.
\end{defn}

By Theorem \ref{t:old4}, uniformly analytic nc germs about $Y$ are in one-to-one correspondence with sequences of multilinear maps $(f_\ell)_\ell$ satisfying $\LAC(Y)$ and $$\limsup_{\ell\to\infty}\sqrt[\ell]{\|f_\ell\|_{\cb}}<\infty.$$
Similarly, analytic nc germs about $Y$ embed into the set of sequences of multilinear maps $(f_\ell)_\ell$ satisfying $\LAC(Y)$ and
$$\limsup_{\ell\to\infty}\sqrt[\ell]{
	\sup_{\|Z\|_{ns}=1}	\|(f_\ell)_n (Z,\dots,Z)\|_{ns}
}<\infty$$
for all $n$ (this embedding is indeed proper, see \cite[Example 8.6]{KVV3}). More generally, a sequence of multilinear maps $(f_\ell)_\ell$ satisfying $\LAC(Y)$ for some $Y\in\mats^g$ is called a {\bf formal nc germ about $Y$}. Formal nc germs are endowed with natural addition and convolution multiplication, thus forming a $\kk$-algebra $\cO_Y$. For $Y\in\matcs^g$ we have
$$\Oua_Y\subset\Oa_Y\subset\cO_Y,$$
and these inclusions are strict; see \cite[Example 8.14]{KVV3} or \cite[Section 17]{Voi}.

In \cite[Chapter 5]{KVV3} it is described in detail how formal nc germs about $Y\in\mats^g$ can be viewed as germs of nc functions. We say that $Z=(Z_1,\dots,Z_g)\in\matn^g$ is {\bf jointly nilpotent} if $Z_1,\dots,Z_g$ generate a nilpotent $\kk$-algebra in $\matn$. Let
$$\Nilp(Y)=\bigsqcup_n \left\{
X\in\mat{ns}^g\colon X-\oplus^n Y \text{ is jointly nilpotent}
\right\}\subset \kk^g_{\nc}.$$
Then $\Nilp(Y)$ is an admissible nc set, so every nc function on $\Nilp(Y)$ admits nc differential operators, which form a formal nc germ. Conversely, every formal nc germ determines an nc function on $\Nilp(Y)$ via Taylor-Taylor series.

\begin{rem}\label{r:scalar}
If $Y\in\kk^g$, then $\cO_Y\cong\ser$. Moreover, for every $Y\in \C^g$ we have $\Oa_Y\cong \Oa_0$ and $\Oua_Y\cong\Oua_0$, where $0=0^g\in\C^g$.
\end{rem}

\subsection{Universal skew field of fractions}

Finally, we review some notions from skew field theory following \cite[Section 7.2]{Coh}.

If $F$ and $E$ are skew fields, then a {\bf local homomorphism} from $F$ to $E$ is given by a ring homomorphism $R_0\to E$, whose domain $R_0\subseteq F$ is a local subring and whose kernel contains precisely the elements that are not invertible in $R_0$.

Let $R$ be a ring. A skew field $\cU$ is a {\bf universal skew field of fractions of $R$} if there is an embedding $R\hookrightarrow \cU$ whose image generates $\cU$ as a skew field, and every homomorphism $R\to D$ into a skew field $D$ extends to a local homomorphism from $\cU$ to $D$ whose domain contains $R$. The universal skew field of fractions is, when it exists, unique up to isomorphism \cite[Section 7.2]{Coh}. It also has an alternative characterization \cite[Theorem 7.2.7]{Coh}: every matrix over $R$, which becomes invertible under some homomorphism from $R$ to a skew field, is invertible over $\cU$.

A ring $R$ is a {\bf semifir} \cite[Section 2.3]{Coh} if every finitely generated left ideal in $R$ is a free left $R$-module of unique rank. Let $D$ be a skew field containing $\kk$. Then $D\otimes\px$ and its completion $\dser{D}$ are well-known examples of semifirs \cite[Corollary 2.5.2; Theorems 2.9.5 and 2.9.8]{Coh}. By \cite[Corollary 7.5.14]{Coh}, every semifir $R$ admits a universal skew field of fractions $\cU$.

Let $R$ be an arbitrary ring and $A\in R^{d\times e}$. The {\bf inner rank} of $A$ is the smallest $r\in\N\cup\{0\}$ such that $A=BC$ for some $B\in R^{d\times r}$ and $C\in R^{r\times e}$; we denote $\rho (A)=r$. Furthermore, $A$ is {\bf full} if $\rho (A)=\min \{d,e\}$, and {\bf non-full} otherwise. These notions give us yet another characterization of the universal skew field of fractions in the case $R$ is a semifir. By \cite[Theorem 7.5.13]{Coh}, the following are equivalent for a skew field $\cU$ containing $R$ and generated by $R$:
\begin{enumerate}
	\item $\cU$ is the universal skew field of fractions of $R$;
	\item the embedding $R\subseteq \cU$ is inner-rank preserving;
	\item every full square matrix over $R$ is invertible over $\cU$.
\end{enumerate}

\section{Universal skew field of fractions of formal power series and the Amitsur-Cohn theorem for meromorphic identities}\label{sec3}

In this section we construct noncommutative formal meromorphic functions as the elements of the universal skew field of fractions of noncommutative formal power series, see Subsection \ref{ssec3.4} and Theorem \ref{t:formal}. Further, a meromorphic variant of Amitsur's theorem \cite[Theorem 16]{Ami} and Cohn's theorem \cite[Theorem 7.8.3]{Coh} for noncommutative rational functions is given in Theorem \ref{t:id}: a formal meromorphic expression vanishes under all finite-dimensional representations if and only if it vanishes in every stably finite algebra.

\subsection{Meromorphic expressions and identities}

The $\kk$-algebra $\ser$ has a natural topology as the completion of $\px$. Let $\cA$ be a unital $\kk$-algebra. Then a homomorphism $\phi:\ser\to \cA$ coinduces a topology on its image, and this topology is Hausdorff if and only if the ideal $\ker\phi$ is closed in $\ser$. Whenever $\cA$ does not have any specified topology, we call $\phi$ {\bf continuous} if $\ker\phi$ is closed in $\ser$. If $s\in\ser$ has homogeneous components $s^{(i)}$, that is,
$$s=\sum_{i=0}^\infty s^{(i)},$$
and $\phi$ is continuous, then $\phi(s)=0$ if and only if $\phi(s^{(i)})=0$ for all $i$.

\begin{defn}\label{d:mero}
A formal {\bf meromorphic expression} over $x$ over $\kk$ is an expression of the form $m=r(s_1,\dots,s_\ell)$, where $r$ is a formal rational expression in the letters $y=(y_1,\dots,y_\ell)$ and $s_1,\dots,s_\ell\in\ser$. Here, a formal rational expression \cite[Section 3]{Ber} is a syntactically valid string made of scalars $\kk$, letters $y$, arithmetic operations $+,\cdot,{}^{-1}$ and parentheses; it can be evaluated on a tuple of elements in a $\kk$-algebra in a natural way provided that all the inverses exist; see \cite[Subsection 2.3]{HMS} for technical details. If $\cA$ is a $\kk$-algebra, then $m$ is
\begin{enumerate}
\item a {\bf meromorphic identity (MI) for $\cA$} if for every continuous homomorphism $\phi:\ser\to \cA$, $\phi(m):=r(\phi(s_1),\dots, \phi(s_\ell))$ is either undefined or 0.
\item a {\bf formal meromorphic identity (FMI) for $\cA$} if for every homomorphism $\phi:\ser\to \cA$, $\phi(m):=r(\phi(s_1),\dots, \phi(s_\ell))$ is either undefined or 0.
\end{enumerate}
\end{defn}

\begin{rem}
The distinction between MI and FMI is required because not every ideal in $\ser$ is closed. For example, let $J$ be the ideal in $\ser$ generated by the commutators $[x_i,x_j]$ for $i,j=1,\dots,g$. Then one can check that
$$\left[x_1,\sum_{i=1}^\infty x_1^ix_2x_1^i  \right]=\sum_{i=1}^\infty x_1^i[x_1,x_2]x_1^i$$
does not belong to $J$, but it lies in the closure of $J$. In particular, $\ser/J$ is not commutative and therefore not isomorphic to $\kk[[x]]$. This differs from the commutative setting, where every ideal in $\kk[[x]]$ is closed \cite[Theorem 8.14]{Mat}.
\end{rem}

\begin{rem}
As opposed to the PI theory, central simple algebras of the same degree do not satisfy the same ``series'' identities. For example, $\kk$ and $\kk((t))$ are both 1-dimensional (commutative) fields; however, there is only one homomorphism $\kk((t))\to \kk$, while there are several homomorphisms $\kk((t))\to\kk((t))$. As a consequence, $\sum_{i\ge 1}x_1^i$ is a MI for $\kk$ but not for $\kk((t))$.
\end{rem}

\subsection{Completion of the ring of generic matrices}

Fix $n\in\N$ and let $\kk((\xi))$ be the field of fractions of the polynomial ring in $gn^2$ variables
$$\kk[[\xi]]=\kk[[\xi_{ij}^k:1\le i,j\le n, 1\le k \le g]].$$
Let $\gX_k=(\xi_{ij}^k)_{i,j}$ be $n\times n$ generic matrices, and let $\gm\subset\opm_n(\kk[\xi])$ be the algebra of generic matrices \cite[Definition 1.3.5]{Row}, i.e., the unital $\kk$-algebra generated by $\gX=(\gX_1,\dots,\gX_g)$. Let $\gms$ be the closure of $\gm$ in $\matser$. Equivalently, $\gms$ is the completion of $\gm$ with respect to the ideal generated by $\gX$, and its elements are formal power series in $\gX$; cf. \cite{GMS} for an analytic tracial version. Since $\matquot$ is clearly a scalar extension of $\gm$ and hence of $\gms$, we conclude that $\gms$ is a prime ring. Its center is thus a domain; let $\uds$ be the ring of central quotients of $\gms$. Since $\gms$ is a PI-ring, Posner's theorem \cite[Theorem 1.7.9]{Row} implies that $\uds$ is a central simple algebra of degree $n$.

\begin{prop}\label{p:sf}
$\uds$ is a skew field.
\end{prop}

\begin{proof}
Suppose that $\uds$ is not a skew field. Since it is a central simple algebra, we conclude that $\gms$ contains nilpotents. Let $f\in\gms$ be nilpotent. Write $f=\sum_{i=d}^\infty f_i$, where $f_i\in \gm$ is homogeneous of degree $i$, and $f_d\neq0$. Then $f^n=0$ implies $f_d^n=0$, which is a contradiction since $\gm$ is a domain.
\end{proof}

The skew field $\uds$ has a special role among the division algebras of degree $n$ (Proposition \ref{p:fd} below), which will be important in subsequent construction of nc germs. First we require the following.

\begin{lem}\label{l:ext}
For $n\in\N$ let $\kk[\xi]\subset R\subset \kk[[\xi]]$ be the ring generated by the entries of elements in $\gms$. Let $C$ be a commutative $\kk$-algebra. Then every continuous homomorphism $\gms\to \opm_n(C)$ extends to a homomorphism $\opm_n(R)\to\opm_n(C)$.
\end{lem}

\begin{proof}
Let $\phi:\gms\to \opm_n(C)$ be a homomorphism. Clearly there is a homomorphism $\phi':\opm_n(\kk[\xi])\to \opm_n(C)$ such that $\phi|_{\gm}=\phi'|_{\gm}$. Since $\opm_n(\kk[\xi])$ is generated by $\gm$ and $\matn$, this implies
\begin{equation}\label{e:prod}
\sum_{i=1}^m a_{i_1}f_{i_1}a_{i_2}\cdots f_{i_{\ell_i}}a_{i_{\ell_i+1}}=0\qquad\Rightarrow\qquad
\sum_{i=1}^m a_{i_1}\phi(f_{i_1})a_{i_2}\cdots \phi(f_{i_{\ell_i}})a_{i_{\ell_i+1}}=0
\end{equation}
for all $a_{i_j}\in\matn$ and $f_{i_j}\in\gm$. Because $\opm_n(R)$ is the $\kk$-subalgebra in $\matser$ generated by $\gms$ and $\matn$, there is a homomorphism $\phi'':\opm_n(R)\to \opm_n(C)$ defined by
$$\sum_{i=1}^m a_{i_1}f_{i_1}a_{i_2}\cdots f_{i_{\ell_i}}a_{i_{\ell_i+1}}\mapsto \sum_{i=1}^m a_{i_1}\phi(f_{i_1})a_{i_2}\cdots \phi(f_{i_{\ell_i}})a_{i_{\ell_i+1}}$$
for all $a_{i_j}\in\matn$ and $f_{i_j}\in\gms$. Indeed, to show that $\phi''$ is well-defined it suffices to verify this on homogeneous elements in $\opm_n(R)$, for which $\phi''$ is well-defined by \eqref{e:prod}. 
\end{proof}

\begin{prop}\label{p:fd}
Let $m$ be a meromorphic expression and $n\in\N$. If $m(\Xi)=0$ in $\uds$, then $m$ is a MI for every division algebra of degree $n$. 
\end{prop}

\begin{proof}
Let $D$ be a division algebra of degree $n$. By applying PI theory to homogeneous components it follows that every continuous homomorphism $\phi:\ser\to D$ factors through $\varphi:\gms\to D$. Let $C$ be a splitting field for $D$ and compose the inclusion $D\hookrightarrow M_n(C)$ with $\varphi$ to obtain $\varphi_1:\gms\to M_n(C)$. Furthermore, $\varphi_1$ extends to $\varphi_2: \opm_n(R)\to\opm_n(C)$ by Lemma \ref{l:ext}, where $R$ is the ring generated by the entries of elements in $\gms$. Let $m=r(s_1,\dots,s_\ell)$ be a meromorphic expression and assume $\phi(m)$ is defined. By induction on the height of $r$ we can assume that $m(\gX)$ is also defined.

By the Cayley-Hamilton theorem there exists a polynomial $p$ in generic matrices and traces of products of generic matrices, and a polynomial $q$ in traces of products of generic matrices, such that the following holds: for every commutative ring $S$ and matrices $A_1,\dots,A_\ell\in\opm_n(S)$ such that all inverses appearing in the evaluation of $r$ at $A=(A_1,\dots,A_\ell)$ exist, then $q(A)$ is invertible in $S$ and $r(A)=q(A)^{-1}p(A)$.

Observe that $p(s_1(\gX),\dots,s_\ell(\gX))\in \opm_n(R)$ and therefore
\begin{align*}
r(\varphi_1(s_1(\gX)),\dots,\varphi_1(s_\ell(\gX)))
&=q(\varphi_1(s_1(\gX)),\dots,\varphi_1(s_\ell(\gX)))^{-1} p(\varphi_1(s_1(\gX)),\dots,\varphi_1(s_\ell(\gX))) \\
&=\varphi_2\big(q(s_1(\gX),\dots,s_\ell(\gX))\big)^{-1}\varphi_2\big(p(s_1(\gX),\dots,s_\ell(\gX))\big)
\end{align*}
by the previous paragraph. Since $m(\gX)=0$ implies $p(s_1(\gX),\dots,s_\ell(\gX))=0$, we have
$$r(\varphi_1(s_1(\gX)),\dots,\varphi_1(s_\ell(\gX)))=0.$$
Consequently,
\[\phi(m)=r(\phi(s_1),\dots,\phi(s_\ell))=r(\varphi(s_1(\gX)),\dots,\varphi(s_\ell(\gX)))=0.\qedhere\]
\end{proof}

\subsection{Construction of the skew field \texorpdfstring{$\cM$}{M} and its universality}\label{ssec3.4}

Every nc power series can be evaluated at a tuple of generic matrices, resulting in a matrix of commutative power series. Likewise, one can evaluate a formal rational expression of nc power series on a tuple of generic matrices, which either yields a matrix of fractions of commutative power series or is undefined due to a matrix singularity at some point of the calculation.

The following type of construction first originated with noncommutative rational functions \cite{HMV}. Let $\cM'$ be the set of formal rational expressions over $\ser$ such that for $m\in\cM'$, $m(\gX)$ is defined for a generic tuple $\gX$ of some size. If $m_1,m_2\in\cM'$, then let $m_1\sim m_2$ if and only if $m_1(\gX)=m_2(\gX)$ for a generic tuple $\gX$ of any size (when both are defined). It is easy to see that $\sim$ is an equivalence relation on $\cM'$. By Proposition \ref{p:sf}, $\cM=\cM'/\!\sim$ is a skew field; the equivalence class of $m$ is denoted $\mm$. If $s\in\ser$, then $s(\gX)=0$ for all sizes of $\gX$ implies $s=0$; hence $\ser$ naturally embeds into $\cM$. Elements of $\cM$ are called {\bf formal meromorphic nc germs}.

\begin{lem}\label{l:main}
Let $m\in\cM'$ be an MI for $\widehat{\operatorname{UD}}_{N}$ for all $N\in\N$. Then $m$ represents 0 in the universal skew field of fractions of $\ser$.
\end{lem}

\begin{proof}
By the assumption, $m(\gX)$ is defined for a tuple of $n\times n$ generic matrices $\gX$. Let $\usf$ and $\usf'$ be universal skew fields of fractions of $\ser$ and $\dser{\uds}$, respectively. Since $m(\gX)\in \uds$ and $\uds$ is a skew field by Proposition \ref{p:sf}, $m$ represents an element in $\usf$ by the definition of $\cU$. We define a homomorphism $\phi:\ser\to \dser{\uds}$ as follows. For $w\in\mx$ consider $w(x+\gX)\in \uds\otimes\px$; we can write it as
$$w(x+\gX)=\sum_{i=1}^{2^{|w|}} u_{w,i}(\gX)\otimes v_{w,i}$$
for $u_{w,i},v_{w,i}\in\mx$. Let $s=\sum_w \alpha_w w\in\ser$. Then for every $v\in\mx$,
$$\sum_{w\colon \exists \iota:v=v_{w,\iota}} \alpha_w \left(\sum_{i\colon v=v_{w,i}}u_{w,i}(\gX) \right)\in \gms$$
because the inner sum is finite and homogeneous, and the outer sum contributes only finitely many terms of a fixed degree. Therefore we can define
$$\phi(s)=
\sum_v\left(\sum_{w\colon \exists \iota:v=v_{w,\iota}} \alpha_w \left(\sum_{i\colon v=v_{w,i}}u_{w,i}(\gX) \right)\right)v
\in\dser{\uds}.$$
It is easy to check that $\phi$ is indeed a homomorphism (although not a continuous one). Because $\phi(m)|_{x=0}=m(\gX)\in \uds$, $\phi(m)$ represents an element in $\usf'$ by the universal property of $\cU'$. Moreover, since $m$ can be evaluated at $\Xi$, all the inverses appearing in $\phi(m)\in\usf'$ already appear in $\dser{\uds}$, so actually $\phi(m)\in\dser{\uds}$.

Next observe that $m$ represents 0 if $\phi(m)$ represents 0. Indeed: consider the continuous homomorphism $\psi:\dser{\gms} \to \ser$ determined by $\psi(\gX_j)=0$ and $\psi(x_j)=x_j$. Since $\usf'$ contains a skew field of fractions of $\dser{\gms}$, by Zorn's Lemma there exists a subring $\dser{\gms}\subseteq L\subseteq\usf'$ maximal with the property that $\psi$ extends to a (not necessarily local) homomorphism $\psi':L\to \usf$. By induction on the inversion height of $m$ we see that $\phi(m)\in L$ and $\psi'(\phi(m))=m$, so $m\neq0$ implies $\phi(m)\neq0$.

Let $n'\in\N$ and $\gX'$ be a tuple of $n'\times n'$ generic matrices. Then there is a continuous homomorphism
$$\dser{\uds}\to \opm_{nn'}\big(\kk((\xi))((\xi'))\big),\qquad S\mapsto S(\gX').$$
By the definition of $\phi$ we have
\begin{equation}\label{e:23}
\phi(m)(\gX')=m(\gX\otimes I+I\otimes \gX').
\end{equation}
Since $\phi(m)\in\dser{\uds}$, we have $\phi(m)=\sum_w q_w w$ for $q_w\in \uds$. Observe that
\begin{equation}\label{e:24}
\frac{{\rm d}}{{\rm d} t^h} \phi(m)\big(t \gX'\big)\Big|_{t=0}=h!\sum_{|w|=h} q_w w(\gX')
\end{equation}
for every $h\in\N\cup\{0\}$. Let $p_h=\sum_{|w|=h} q_w w\in \uds\otimes\px$.

Under the natural inclusion $\kk((\xi,\xi'))\subset\kk((\xi))((\xi'))$ we see that
\begin{equation}
m(\gX\otimes I+I\otimes \gX')\in \opm_{nn'}\big(\kk((\xi,\xi'))\big).
\end{equation}
Since the homomorphism
$$\ser\to \opm_{nn'}\big(\kk[[\xi,\xi']]\big),\qquad s\mapsto s(\gX\otimes I+I\otimes \gX')$$
is continuous with respect to the natural topology on $\opm_{nn'}(\kk[[\xi,\xi']])$ and $m$ is an MI for $\widehat{\operatorname{UD}}_{nn'}$, we have $m(\gX\otimes I+I\otimes \gX')=0$ by assumption. Therefore $p_h(\gX')=0$ for every $n'\in\N$ by \eqref{e:23} and \eqref{e:24}, and consequently $p_h(X)=0\in\opm_{n'}(\uds)$ for every $X\in \mat{n'}^g$. As in the proof of \cite[Lemma 1.4.3]{Row}, we can use a ``staircase'' of standard matrix units to show that $p_h=0$. Hence $\phi(m)=0$ and thus $m$ represents 0 in $\usf$.
\end{proof}

\begin{thm}\label{t:formal}
$\cM$ is the universal skew field of fractions of $\ser$.
\end{thm}

\begin{proof}
Let $\usf$ be the universal skew field of fractions of $\ser$. By the universality there exists a local homomorphism from $\cU$ to $\cM$. That is, there is a subring $\ser\subseteq L\subseteq\usf$ and a homomorphism $\phi:L\to\cM$ extending the inclusion $\ser\subset \cM$ such that $\phi(u)\neq0$ implies $u^{-1}\in L$. It suffices to prove that $\ker\phi=0$.

Let $m$ be a meromorphic expression representing an element of $L$, and suppose $\phi(m)=0$. Since $\phi$ extends the inclusion $\ser\subset\cM$, we have $m\in\cM'$ and $m(\gX)=0$ for every generic tuple $\gX$ (if defined) by the construction of $\cM$. By Proposition \ref{p:fd}, $m$ is an MI for $\uds$ for all $n\in\N$. Therefore $m$ represents 0 in $\cU$ by Lemma \ref{l:main}, so $\ker\phi=0$.
\end{proof}

\subsection{Amitsur-Cohn theorem for meromorphic identities}

An algebra $\cA$ is {\bf stably finite} (or weakly finite) if for every $n\in\N$ and $A,B\in \cA^{n\times n}$, $AB=I$ implies $BA=I$; see e.g. \cite[Chapter 5]{RLL} and \cite[Section V.2]{Bla} for analytic examples. The following result is a meromorphic fusion of theorems on rational identities by Amitsur \cite[Theorem 16]{Ami} and Cohn \cite[Theorem 7.8.3]{Coh}.

\begin{thm}\label{t:id}
Let $m$ be a meromorphic expression. Then the following are equivalent:
\begin{enumerate}
	\item $m\notin\cM'$ or $\mm=0\in\cM$;
	\item $m$ is an MI for $\uds$ for all $n\in\N$;
	\item $m$ is an FMI for every stably finite algebra.
\end{enumerate}
\end{thm}

\begin{proof}
$(3)\Rightarrow(2)$ is trivial since every skew field is stably finite, and $(2)\Rightarrow(1)$ follows by the construction of $\cM$.

$(1)\Rightarrow(3)$ Let $m=r(s_1,\dots,s_\ell)$. By \cite[Theorem 4.12]{HMS} there exist $Q\in\py^{d\times d}$ and $u,v\in\kk^d$ satisfying the following: for every $\kk$-algebra $\cB$ and $b\in \cB^\ell$ such that $r(b)$ exists, $Q(b)$ is invertible over $\cB$ and $r(b)=v^t Q(b)^{-1}u$. Note that
\begin{equation}\label{e:full}
\begin{pmatrix}1 & v^t Q^{-1}(b) \\ 0 & I\end{pmatrix}
\begin{pmatrix}-r(b) & 0 \\ 0 & I\end{pmatrix}
\begin{pmatrix}1 & 0 \\ u & Q(b)\end{pmatrix}
=
\begin{pmatrix}0 & v^t \\ u & Q(b)\end{pmatrix}=:A(b).
\end{equation}
If $\cB$ is stably finite, then $r(b)\neq0$ implies that $A(b)$ is a full matrix by \cite[Proposition 0.1.3]{Coh}.

Now let $\cA$ be a stably finite algebra and $\phi:\ser\to \cA$ a homomorphism such that $\phi(m)$ is well-defined and nonzero. Since $Q(\phi(s_1),\dots,\phi(s_\ell))$ is invertible over $\cA$, $Q(s_1,\dots,s_\ell)$ is full over $\ser$. Since $\ser$ is a semifir and $\cM$ is its universal skew field of fractions by Theorem \ref{t:formal}, $Q(s_1,\dots,s_\ell)$ is invertible over $\cM$. Therefore $m\in\cM'$. Furthermore, since $A(\phi(s_1),\dots,\phi(s_\ell))$ is full over $\cA$, $A(s_1,\dots,s_\ell)$ is full over $\ser$. As before, $A(s_1,\dots,s_\ell)$ is invertible over $\cM$. Therefore $\mm=r(s_1,\dots,s_\ell)$ is nonzero in $\cM$ by \eqref{e:full}.
\end{proof}

\begin{rem}
Formal expressions involving inverses behave pathologically for algebras that are not stably finite. For example, take $\cA=\cB(\ell^2(\N))$, $m=x_1(x_2x_1)^{-1}x_2-1$ and $X=(S,S^*)$, where $S$ is the right shift operator on $\ell^2(\N)$. Then $m$ is a rational identity but $m(X)\neq0$.
\end{rem}

\section{Meromorphic \texorpdfstring{$\glc$}{GLn(C)}-invariants}\label{sec4}

As nc functions respect similarities, invariant theory plays an important role in free analysis \cite{KS}. Let $n\in\N$ and consider the action of $\glc$ on $\matcn^g$ given by $X^a= aXa^{-1}$ for $X\in\matcn^g$ and $a\in\GL_n(\C)$. A map $f:\matcn^g\to\matcn$ is a $\glc$-{\bf concomitant} (or an {\bf equivariant map}) if it intertwines with the action of $\glc$ on $\matcn^g$ and $\matcn$. In parallel with the classical invariant theory, where $\ud$ is identified with the ring of rational concomitants \cite{Pro,Sal}, we relate $\uds$ with meromorphic concomitants; see \cite{Lum,GMS} for analytic concomitants.

Consider the action of $\glc$ on $\matquotc$ given by
\begin{equation}\label{e:act}
f^a=a^{-1}f(a\xi a^{-1})a
\end{equation}
for $f\in \matquotc$ and $a\in\GL_n(\C)$. Then $f$ is invariant for this action if and only if it is a $\glc$-concomitant. Observe that this action preserves $\matserc$ and its homogeneous components. By \cite[Theorem 2.1]{Pro} it follows that elements of $\matserc^{\glc}$ are power series in products of words and traces of words in the tuple of generic matrices $\gX$.

We say that $f\in\C[[\xi]]$ is analytic if it converges absolutely and uniformly on some neighborhood of $0\in\C^{gn^2}$. Let $O\subset\C[[\xi]]$ be the subring of analytic series, and let $M$ be its field of fractions. Let $\uc\subset\glc$ be the unitary group.

\begin{lem}\label{l:rey}
Let $f\in \C[[\xi]]$ and $f(0)=0$. If $f$ divides $f^a$ for every $a\in\uc$, then $f=\tilde{f}h$ for some $h\in \C[[\xi]]^*$ and $\tilde{f}\in \C[[\xi]]^{\glc}$.

Moreover, if $f\in O$, then one can choose $\tilde{f},h\in O$.
\end{lem}

\begin{proof}
Write $f=\sum_{i=d}^\infty f_i$ with $d\ge1$ and $f_i$ homogeneous of degree $i$. Consider the map
$$\lambda:\uc\to\C^*,\qquad \lambda(a)=\left(\frac{f^a}{f}\right)(0).$$
Since
$$\left(\frac{f^{ab}}{f}\right)(0)
=\left(\frac{f^{ab}}{f^b}\frac{f^b}{f}\right)(0)
=\left(\left(\frac{f^a}{f}\right)^b\right)(0)\left(\frac{f^b}{f}\right)(0)
=\left(\frac{f^a}{f}\right)(0)\left(\frac{f^b}{f}\right)(0),
$$
$\lambda$ is a continuous group homomorphism. As every 1-dimensional representation of $\uc$ factors through the determinant, we have $\lambda(a)=\det(a)^t$ for some integer $t$. By \eqref{e:act} we see that $\ker\lambda$ contains all scalar matrices, so $t=0$ and $\lambda=1$.

For every $a\in\uc$ there exist homogeneous polynomials  $h_{a,\ell}$ of degree $\ell$ for $\ell\in\N$ such that
$$f^a=\left(1+\sum_{i\ge1}h_{a,i}\right)f.$$
For each $\ell\ge1$ we thus have
\begin{equation}\label{e:ell}
h_{a,\ell}f_d=f_{d+\ell}^a-f_{d+\ell}-\sum_{i=1}^{\ell-1}h_{a,i}f_{d+\ell-i}.
\end{equation}
By induction on $\ell$ we see from \eqref{e:ell} that the map $a\mapsto h_{a,\ell}$ from $\uc$ to the space of homogeneous polynomials of degree $\ell$ is continuous with respect to the Euclidean topology. Hence we can define homogeneous polynomials of degree $\ell$
$$k_\ell=\int h_{a,\ell}\, d\mu(a),$$
where $\mu$ is the (right) Haar measure on $\uc$. Let
$$\tilde{f}=\sum_{i=d}^\infty\left(\int f_i^a\, d\mu(a)\right),\qquad h=1+\sum_{i\ge1} k_i.$$
By \eqref{e:ell} we have $hf=\tilde{f}$ and $\tilde{f}$ is $\uc$-invariant by construction. Furthermore, $\uc$ is Zariski dense in $\glc$ by the unitarian trick \cite[Corollary 8.6.1]{Pro1}, so $\tilde{f}$ is also $\glc$-invariant.

Now suppose $f$ is analytic. Then there is a neighborhood $D$ of the origin such that $f^a$ converges absolutely and uniformly on $D$ for all $a\in\uc$. For every $\alpha\in D$ we have
$$\tilde{f}(\alpha)=\int f^a(\alpha)\, d\mu(a),$$
so $\tilde{f}$ also converges absolutely and uniformly on $D$. Since $fh-\tilde{f}=0$ and $f,\tilde{f}$ are analytic, $h$ is also analytic, e.g. by Artin's approximation theorem \cite[Theorem 1.2]{Art}.
\end{proof}

\begin{lem}\label{l:13}
If $r\in \matquotc^{\glc}$, then $r=f_0/f$ for some $f_0\in\matserc^{\glc}$ and $f\in \C[[\xi]]^{\glc}$.

Moreover, if $r\in\opm_n(M)^{\glc}$, then one can choose $f_0\in\opm_n(O)^{\glc}$ and $f\in O^{\glc}$.
\end{lem}

\begin{proof}
Because $\matquotc$ has a $\glc$-invariant basis $\{\gX_1^i\gX_2^j\colon 1\le i,j\le n\}$, it suffices to assume $r\in \C((\xi))^{\glc}$.
Since $\C[[\xi]]$ is a unique factorization domain, we can write $r=f_0/f$ for some coprime $f_0,f\in\C[[\xi]]$. If $r\notin \C[[\xi]]^{\glc}$, then $f(0)=0$. For every $a\in\glc$ we have $r^a=r$ and hence $f_0^af=f^af_0$, so $f$ divides $f^a$. By Lemma \ref{l:rey} there exist $h\in\C[[\xi]]^*$ and $\tilde{f}\in\C[[\xi]]^{\glc}$ such that $f=h\tilde{f}$. Then $r=(h^{-1}f_0)/\tilde{f}$ and $h^{-1}f_0\in\C[[\xi]]^{\glc}$.

Exactly the same reasoning applies in the analytic situation since $O$ is also a unique factorization domain \cite[Proposition 6.4.9]{Kra}.
\end{proof}

\begin{prop}\label{p:denom}
There exists a homogeneous polynomial in the center of $\gm$ such that every $s\in \matserc^{\glc}$ can be written as $s = p^{-1}q$ for some $q\in\gms$. If $s\in \opm_n(O)^{\glc}$, then $q\in \gms\cap \opm_n(O)$.
\end{prop}

\begin{proof}
Let $\tn$ be the subalgebra of $\opm_n(\C[\xi])$ generated by $\gm$ and $\tr(\gm)$. Then $\opm_n(\C[\xi])^{\glc}=\tn$ by \cite[Theorems 1.3 and 2.1]{Pro}. By \cite[Theorem 4.3.1]{Row} there is a multilinear polynomial $h$ in $n^2$ variables that is a central polynomial for $\gm$ and $h(\gm,\dots,\gm)\gm$ is an ideal in $\tn$. Since $h$ is central for $\gm$, there exist homogeneous $r_1,\dots,r_{n^2}\in\gm$ such that $p:=h(r_1,\dots,r_{n^2})\neq0$. Therefore $p$ is homogeneous, lies in the center of $\gm$ and $p\tn\subset \gm$.

If $s\in \matserc^{\glc}$, then its homogeneous components $s_i$ are also $\glc$-invariant and thus belong to $\tn$. Therefore $s_i=p^{-1}q_i$ for some homogeneous $q_i\in\gm$. Hence $q=\sum_{i\ge0}q_i\in\gms$ and $s=p^{-1}q$.

Furthermore, if entries of $s$ are analytic, then so are the entries of $q=ps$.
\end{proof}

\begin{thm}\label{t:inv}
$\matquotc^{\glc}=\uds$. Moreover, every meromorphic $\glc$-con\-co\-mitant equals $f_1(\gX)f_2(\gX)^{-1}$ for some analytic power series $f_1(\gX),f_2(\gX)$ in $\gX$.
\end{thm}

\begin{proof}
Clearly $\uds\subseteq \matquotc^{\glc}$ holds since $\uds$ is the ring of central quotients of $\gms$. Conversely, every $\glc$-invariant $r\in\matquotc$ can be written as $r=f_0/f$ for $f_0\in\matserc^{\glc}$ and $f\in \C[[\xi]]^{\glc}$ by Lemma \ref{l:13}, so $r\in\uds$ by Proposition \ref{p:denom}.

The second statement follows in same manner by the analytic parts of Lemma \ref{l:13} and Proposition \ref{p:denom}.
\end{proof}

The results of this section complement the earlier research of Luminet \cite{Lum} concerning the rings of analytic concomitants on matrices of \emph{fixed size $n$}. Therein it is shown that meromorphic concomitants arising from germs about an irreducible point $X\in\matc{n}^g$ (see Section \ref{sec6} for the definition of irreducibility) are isomorphic to $\opm_n(K)$, where $K$ is the field of fractions of analytic commutative power series in $(g-1)n^2+1$ indeterminates \cite[Propositions 5.3 and 6.2]{Lum}. On the other hand, the results of this section connect meromorphic concomitants arising from germs about the origin (which is far from being irreducible) with the restrictions of nc meromorphic germs to $n\times n$ matrices. The rest of this papers deals with analytic nc germs in the size-independent setting.

\section{Universal skew field of fractions of analytic germs}\label{sec5}

In this section we show that the ring of (uniformly) analytic germs about a scalar point $Y\in\C^g$ admits a universal skew field of fractions, which we call the skew field of nc (uniformly) meromorphic germs; see Subsection \ref{ssec5.2}. This theory is used in the local-global rank principle, Theorem \ref{t:rank}, to relate the intrinsic rank of matrices over $\Oa_Y$ or $\Oua_Y$ with the ranks of their matrix evaluations. We note that there is no commutative analog of this statement.

In Section \ref{sec6} we will see that for $Y\in\matcs^g$ with $s\ge2$, the algebras of germs about $Y$ depend strongly on $Y$. However, for every $Y\in\C^g$ we have
$$\Oua_Y\cong\Oua_0,\qquad \Oa_Y\cong\Oa_0,\qquad \Oua_0\subset\Oa_0\subset\serc.$$
Therefore we can without loss of generality assume $Y=0$.

A (not necessarily commutative) ring $R$ is {\bf local} if it has a unique maximal one-sided ideal $\mathfrak{m}$ \cite[Section 19]{Lam}; in this case, $\mathfrak{m}$ is two-sided and $R/\mathfrak{m}$ is a division ring.

\begin{lem}\label{l:loc}
$\Oa_0$ and $\Oua_0$ are local rings; in both cases, the maximal ideal consists of functions vanishing at the origin.
\end{lem}

\begin{proof}
The statement clearly holds for $\Oa_0$ because commutative analytic germs form a local ring. Now let $f\in \Oua_0$ be such that $f(0)\neq0$. Since $f$ is continuous with respect to the uniformly open topology, there exists $\ve>0$ such that $\|f(0)-f(X)\|<\frac{1}{2|f(0)|}$ for all $X\in\BB_\ve(0)$. Then
$$f^{-1}(X)=f(0)^{-1}\sum_{\ell=0}^\infty \left(\frac{f(0)-f(X)}{f(0)} \right)^\ell$$
is an nc function that converges absolutely and uniformly on $\BB_\ve(0)$, so $f^{-1}\in\Oua_0$.
\end{proof}

\subsection{Semifir property and inertness}

For $j=1,\dots,g$ define linear operators
$$L_j:\serc\to\serc,\qquad L_j\left(\sum_{w\in\mx} \alpha_w w\right)=\sum_{w\in\mx} \alpha_{x_jw} w.$$
A composite of $L_j$'s is called a {\bf right transduction} \cite[Section 2.5]{Coh} or a {\bf left backward shift} \cite[Section 4.2]{KVV2}.

\begin{lem}\label{l:trans}
	Right transductions preserve $\Oa_0$ and $\Oua_0$.
\end{lem}

\begin{proof}
	If $f\in\Oa_0$ and $\alpha\in\C^g$, then
	\begin{equation}\label{e:shift}
	f\begin{pmatrix} X & 0 \\ \alpha & 0\end{pmatrix}=
	\begin{pmatrix}f(X) & 0 \\ \sum_j \alpha_jL_j(f)(X) & f(0)\end{pmatrix}
	\end{equation}
	whenever defined. Since $f$ is analytic on some open nc neighborhood of $0$, \eqref{e:shift} implies $L_j(f)\in\Oa_0$ for every $j$. The same reasoning applies to $\Oua_0$.
\end{proof}

An embedding of rings $R\subset S$ is {\bf totally inert} \cite[Section 2.9]{Coh} if for every $d\in\N$, $U\subset S^{1\times d}$, $V\subset S^{d\times 1}$ satisfying $UV\subset R$ there exists $P\in\GL_d(S)$ such that for $u\in UP^{-1}$ and $1\le i\le d$, either $u_i\in R$ or $v_i=0$ for all $v\in PV$; and analogously for $v\in PV$.

\begin{prop}\label{p:inert1}
$\Oa_0$ and $\Oua_0$ are semifirs. Moreover, the embeddings $\Oa_0\subset\serc$ and $\Oua_0\subset\serc$ are totally inert.
\end{prop}

\begin{proof}
Since $\Oa_0$ and $\Oua_0$ are local rings whose invertible elements are precisely functions non-vanishing at the origin, and right transductions preserve $\Oa_0$ and $\Oua_0$ by Lemma \ref{l:trans}, they are semifirs by \cite[Proposition 2.9.19]{Coh}. In particular, they are semihereditary rings \cite[Section 2.1]{Coh}, their maximal ideals $\mathfrak{m}$ are finitely generated as right ideals, and $\bigcap_n\mathfrak{m}^n=0$. Therefore $\Oa_0\subset\serc$ and $\Oua_0\subset\serc$ are totally inert embeddings by \cite[Corollary 2.9.17]{Coh}.
\end{proof}

\subsection{Meromorphic noncommutative germs}\label{ssec5.2}

Next we construct universal skew fields of fractions $\Ma_0$ and $\Mua_0$ of $\Oa_0$ and $\Oua_0$, respectively. Note that we already know they exist since $\Oa_0$ and $\Oa_0$ are semifirs. Since these constructions are nearly identical, we consider in detail only the case of analytic germs.

One can consider evaluations of formal rational expressions of analytic germs on tuples of matrices near the origin. Let ${\Ma_0}'$ be the set of those expressions that are defined at some tuple of matrices. Observe that if $m\in{\Ma_0}'$ is well defined at some $X\in\matcn^g$, then the restriction of $m$ to $\matcn^g$ is an $n\times n$ matrix of commutative meromorphic functions whose numerators and denominators are analytic about the origin. Then we impose a relation $\sim$ on ${\Ma_0}'$ such that $m_1\sim m_2$ if $m_1(X)=m_2(X)$ for all $X$ in some neighborhood of the origin where $m_1(X)$ and $m_2(X)$ are defined. By analyticity we see that $\sim$ is a well-defined equivalence relation; the equivalence classes are called {\bf meromorphic nc germs}. Since meromorphic commutative germs embed into the field of fractions of commutative power series, meromorphic nc germs form a skew field by Proposition \ref{p:sf}; we denote it $\Ma_0$.

By Proposition \ref{p:inert1}, $\Oa_0\subset\serc$ is a totally inert embedding, and therefore an honest embedding \cite[Section 5.4]{Coh}. Since a universal skew field of fractions of a semifir is determined by full matrices over the semifir, we conclude that the rational closure of $\Oa_0$ in $\cM$ is a universal skew field of fractions of $\Oa_0$. By comparing equivalence relations used to define $\Ma_0$ and $\cM$ it is clear that this rational closure is precisely $\Ma_0$. Therefore we proved the following.

\begin{cor}\label{c:univ}
$\Ma_0$ (resp. $\Mua_0$) is a universal skew field of fractions of $\Oa_0$ (resp. $\Oua_0$).
\end{cor}

Since $\Oa_0$ (resp. $\Oua_0$) is a semifir, every element of $\Ma_0$ (resp. $\Mua_0$) can be represented as
\begin{equation}\label{e:lin}
u^tQ^{-1}v
\end{equation}
for some $u,v\in\C^d$ and a full matrix $Q\in(\Oa_0)^{d\times d}$ (resp. $Q\in(\Oua_0)^{d\times d}$) as a consequence of \cite[Corollary 7.5.14]{Coh}. As we will see in Theorem \ref{t:rank} below, such a $Q$ is of full rank if and only if $Q(X)$ is invertible for some $X\in\C_{\nc}^g$ in a (uniformly) open neighborhood of the origin. If (uniformly) analytic nc germs are given in the form \eqref{e:lin}, then their arithmetic operations can be defined in the same way as for realizations (or linear representations) of nc rational functions \cite{CR,BGM,Vol,HMS}. In the case of nc rational functions, $Q$ can be chosen to be affine, and realizations \eqref{e:lin} with an affine $Q$ of minimal size $d$ exhibit good properties: they are efficiently computable, essentially unique, and the domain of an nc rational function is given as the invertibility set of $Q$. On the other hand, in the (uniformly) analytic case it is unclear whether any of these properties carry over.

\def\D{\mathbb{D}}

Next we show that evaluations of uniformly meromorphic nc germs make sense in arbitrary stably finite Banach algebras, e.g. $C^*$-algebras with a faithful trace.

Let $h:(\C^g)^{\otimes \ell}\to\C$ be a linear map. Then $\|h\|_{\cb}$ is the norm of this functional with respect to the Haagerup norm on $(\C^g)^{\otimes \ell}$; see \cite[Chapter 17]{Pau} or \cite[Chapter 5]{Pis}. Now let $\cA$ be a Banach algebra. If $\mu:\cA^{\otimes\ell}\to\cA$ is given by $\mu(a_1\otimes\cdots\otimes a_\ell)=a_1\cdots a_\ell$, then $\|\mu\|=1$, where $\cA^{\otimes\ell}$ is endowed with the projective cross norm \cite[Section 2.1]{Rya}. Let
$$h^{\cA}:(\cA^g)^\ell\to
(\C^g)^{\otimes \ell}\otimes \cA^{\otimes \ell}\xrightarrow{h\otimes\mu}
\C\otimes\cA=\cA.$$
If $\cA^g$ is endowed with the $\ell_\infty$ norm (with respect to the norm on $\cA$) and $(\cA^g)^{\otimes \ell}$ is endowed with the projective cross norm, then $(\cA^g)^{\otimes \ell}\to
(\C^g)^{\otimes \ell}\otimes \cA^{\otimes \ell}$ is a contraction. Hence the $\ell$-linear map $h^{\cA}$ satisfies
\begin{equation}\label{e:a}
\|h^{\cA}(a_1,\dots,a_\ell)\|\le \|h\|_{\cb}\|a_1\|\cdots\|a_\ell\|
\end{equation}
for $a_i\in\cA$.

For $f\in\Oua_0$ denote
$$\delta(f):=\frac{1}{\limsup_{\ell\to\infty}\sqrt[\ell]{\|\Delta_0^\ell f\|_{\cb}}}>0.$$
If $\cA$ is a Banach algebra, then by applying \eqref{e:a} to $h=\Delta_0^\ell f$ we see that $f$ converges absolutely and uniformly on
$$\left\{a\in\cA^g\colon \|a\|<\delta(f)\right\}.$$

\begin{cor}\label{c:sf}
Let $m$ be a meromorphic expression built of uniformly analytic germs about the origin. If $m$ represents $0$ in $\Mua_0$, then there exists $\ve>0$ such that for every stably finite Banach algebra $\cA$ and $X\in\cA^g$ satisfying $\|X\|<\ve$, $m(X)$ is either undefined or $m(X)=0$.
\end{cor}

\begin{proof}
Let $m=r(s_1,\dots,s_\ell)$ for $s_k\in\Oua_0$. As in the proof of Theorem \ref{t:id}, there exists $A\in\pyc^{d\times d}$ satisfying: for every algebra $\cB$ and $b\in \cB^\ell$ such that $r(b)$ exists,
\begin{equation}\label{e:full1}
r(b)\oplus I = PA(b)Q
\end{equation}
for some invertible matrices $P,Q$ over $\cB$. Moreover, if $\cB$ is stably finite and $A(b)$ is not full, then $r(b)=0$ by \cite[Proposition 0.1.3]{Coh}. If $m$ represents $0$ in $\Mua_0$, then $A(s_1,\dots,s_\ell)$ is not invertible over $\Mua_0$, so $A(s_1,\dots,s_\ell)$ is non-full over $\Oua_0$ by Proposition \ref{p:inert1} and Corollary \ref{c:univ}. Then there is $e<d$ such that $A(s_1,\dots,s_\ell)=BC$ for some matrices $B$ and $C$ over $\Oua_0$ of dimensions $d\times e$ and $e\times d$, respectively. Let $$\ve:=\min\{\delta(s_k),\delta(B_{ij}),\delta(C_{ij})\colon i,j,k\}>0.$$

Let $X\in\cA^g$ be such that $\|X\|<\ve$ and $m(X)$ is defined. Then $s_k$ and the entries of $B,C$ converge at $X$. Since $(A\circ s)(X)=B(X)C(X)$ is not full, $m(X)=0$ by the previous paragraph.
\end{proof}

\begin{rem}
Let $\Omega\ni0$ be an nc set, open and connected in the disjoint union (uniformly open) topology. Then the (uniformly) analytic nc functions on $\Omega$ embed into $\Oa_0$ ($\Oua_0$), so they generate a skew field of fractions inside $\Ma_0$ ($\Mua_0$), whose elements deserve to be called (uniformly) meromorphic nc functions; cf. \cite[Section 10]{AM0}.
\end{rem}

\subsection{Local-global rank principle}

As a consequence of our construction of the universal skew fields of $\Oa_0$ and $\Oua_Y$ we obtain the following theorem relating the inner rank of a matrix over a germ algebra with the maximal ranks of its evaluations on a neighborhood of the origin.

\begin{thm}\label{t:rank}
The inner rank of a matrix $A$ over $\Oa_0$ (resp. $\Oua_0$) equals
\begin{equation}\label{e:rank}
\max\left\{
\frac{\rk A(X)}{n}\colon n\in\N,\ X\in\matcn^g\text{ \rm in a neighborhood of $0$}
\right\}.
\end{equation}
\end{thm}

\begin{proof}
Let $A$ be a $d\times e$ matrix and let $r$ denote \eqref{e:rank}. Clearly we have $\rho (A)\ge r$. Without loss of generality assume $d\le e$.

First we deal with the case $\rho (A)=d$, i.e., $A$ is a full matrix. Since $\Oa_0$ is a semifir and $\Ma_0$ is its universal skew field of fractions by Corollary \ref{c:univ}, $A$ has full rank over $\Ma_0$. Therefore there exists a $d\times (d-e)$ matrix $A'$ over $\Ma_0$ such that $(A\ A')$ is invertible, so there is a $d\times d$ matrix $B$ such that
$(A\ A')B=I$. By the construction of $\Ma_0$, there exists $n\in\N$ such that each entry of $A'$ and $B$ is well-defined at some tuple of $n\times n$ matrices. Furthermore, when restricted to $\matcn^g$, $A'$ and $B$ are matrices of commutative meromorphic functions, so there exists $X\in\matcn$ such that $A'(X),B(X)$ are well-defined. Therefore $(A\ A')(X)B(X)=I$ implies $\rk A(X)=dn$ and hence $r=\rho (A)$.

Now suppose $\rho (A)<d$. Then there exist full matrices $B$ and $C$ over $\Ma_0$ of dimensions $d\times \rho (A)$ and $\rho (A)\times e$, respectively, such that $A=BC$. By the previous paragraph there exist $X\in\matc{m}^g$ and $Y\in\matc{n}^g$ such that $\rk B(X)=m\rho (A)$ and $\rk C(Y)=n\rho (A)$. Then $\rk B(\oplus^n X)=(m+n)\rho (A)=\rk C(\oplus^mY)$. Since the restrictions of $B$ and $C$ to $\matc{m+n}^g$ are matrices of commutative meromorphic functions, there exists $Z\in\matc{m+n}^g$ such that $\rk B(Z)=(m+n)\rho (A)=\rk C(Z)$. Since $d>\rho (A)$, we have $\ker B(Z)=\{0\}$ and therefore $\rk A(Z)=(m+n)\rk C(Z)=(m+n)\rho (A)$. Hence $r=\rho (A)$.
\end{proof}

\begin{rem}
There is no commutative analog of Theorem \ref{t:rank}. Consider the matrix
$$A=\begin{pmatrix}0 & t_3 & -t_2 \\ -t_3 & 0 & t_1 \\ t_2 & -t_1 & 0\end{pmatrix} \in\kk[t_1,t_2,t_3]^{3\times 3}$$
from \cite[Section 5.5]{Coh}. Its inner rank over $\kk[[t_1,t_2,t_3]]$ equals 3. Indeed, suppose that $A=BC$ for $B\in \kk[[t_1,t_2,t_3]]^{3\times 2}$ and $C\in \kk[[t_1,t_2,t_3]]^{2\times 3}$. Then $B(0)C(0)=0$ because $A$ is linear, so at least one of the scalar matrices $B(0),C(0)$ is of rank at most 1. Without loss of generality let $\rk B(0)\ge 1$. Then there exist $U\in \kk^{2\times 3}$ of full rank and $v\in \kk^3\setminus\{0\}$ such that $UB(0)=0$ and $C(0)v=0$. Then $A=BC$ and linearity of $A$ imply $UAv=0$. However, a short calculation shows that this is impossible.

On the other hand, we claim that
$$\frac{\rk A(T)}{n}\le \frac{54}{19}<3$$
for all $n\in\N$ and triples $T\in\matn^3$ of commuting matrices (this bound is presumably not optimal). Write $\kappa=\frac{9}{19}$ and let $T$ be an arbitrary triple of commuting $n\times n$ matrices. If $\kappa n\le \rk T_1+\rk T_2+\rk T_3$, then $\max_i \rk T_i\ge \frac{\kappa n}{3}$. Since
$$A(T)\begin{pmatrix}
T_1 \\ T_2 \\ T_3
\end{pmatrix}=0,$$
it follows that $\rk A(T)\le 3n-\frac{\kappa n}{3}=\frac{54}{19}n$. Now suppose $\kappa n\ge \rk T_1+\rk T_2+\rk T_3$. Then $\dim\ker T_i\ge(1-\kappa)n$ for all $i$, so
$$\dim\left(\ker T_i\cap\ker T_j\right)\ge2(1-\kappa)n-n=(1-2\kappa)n.$$
Since
$$A(T)\supseteq \begin{pmatrix}
\ker T_2\cap \ker T_3 \\ \ker T_1\cap \ker T_3 \\ \ker T_1\cap\ker T_2
\end{pmatrix},$$
it follows that $\rk A(T)\le 3n-3(1-2\kappa)n=\frac{54}{19}n$.
\end{rem}

\begin{cor}
If $A$ is a matrix over $\ser$, then its inner rank over $\ser$ equals
$$\max_n\frac{\rk A(\gX^n)}{n}$$
where $\rk A(\gX^n)$ is the rank of $A(\gX^n)$ in $\matquot$.
\end{cor}

\begin{proof}
Apply the same arguments as in the proof of Theorem \ref{t:rank}.
\end{proof}

\subsection{Level-wise meromorphic functions}

One might be tempted to assert that every uniformly analytic level-wise meromorphic function is an element of $\Ma_0$. However, in this subsection we provide an example of an nc function $f$ with the following properties:

\begin{enumerate}
	\item $f$ is defined on $\C^3_{\nc}\cap\{\det X_3\neq0 \}$ and uniformly bounded on some nc ball about every point therein;
	\item $f$ is level-wise rational; that is, when restricted to $\matcn^3$, $f$ equals $p_n/q_n$ for a matrix polynomial $p_n$ and a scalar polynomial $q_n$;
	\item $f\notin \Mua_0$.
\end{enumerate}

For $s\in\N$ let
$$h_s=\sum_{\pi\in S_{s+1}} \operatorname{sign}(\pi)
x_1^{\pi(1)-1} x_2 x_1^{\pi(2)-1}x_2\cdots x_1^{\pi(s+1)-1}x_2.
$$
Then $h_s$ is a homogeneous polynomial of degree $\frac{s(s+1)}{2}+s$ with $(s+1)!$ terms and $h_s$ vanishes on $\matcs^2$ by \cite[Proposition 1.1.33]{Row}. Define
$$f=\sum_{s=1}^\infty \frac{1}{(s+1)!(s^2)!} h_s(x_1,x_2) x_3^{-s}.$$
Then $f$ is an nc function on $\C^3_{\nc}\cap\{\det X_3\neq0 \}$ and for $X\in\matcn^3$,
\begin{equation}\label{e:ex}
f(X)=\sum_{s=1}^{n-1} \frac{1}{(s+1)!(s^2)!} h_s(X_1,X_2)X_3^{-s}.
\end{equation}
Note that the denominator of \eqref{e:ex} is a homogeneous scalar polynomial of degree $n(n-1)$. The factor $(s+1)!(s^2)!$ ensures uniform convergence.

Now let $\mm\in\Mua_0$ be arbitrary. By the induction on the inversion height of $\mm$ it is easy to see that there exists $d\in\N$ such that the denominator of $\mm$ restricted to $\matcn^3$ has order at most $dn$, for all $n\in\N$. Therefore $f\notin \Mua_0$.

\section{Germs about semisimple points and Hermite interpolation}\label{sec6}

In this section we turn our attention to germs about semisimple (non-scalar) points. We establish a noncommutative Hermite interpolation result, Theorem \ref{t:hermite}, which states that values and finitely many differentials of an arbitrary nc function at a finite set of semisimple points can be interpolated by an nc polynomial. Furthermore, we identify $\cO_Y$ as an inverse limit $\cO_Y = \varprojlim_\ell \big(\px/\cI(Y)^\ell\big)$ in Corollary \ref{c:compl}, where $\cI(Y)$ is the vanishing ideal at $Y$. Lastly, we provide a criterion for distinguishing the germ algebras $\cO_Y$, $\Oa_Y$ and $\Oua_Y$ up to isomorphism with respect to $Y$ (Theorem \ref{t:iso}).

For $Y\in\mats$ let
$$\cI_\ell(Y) =\left\{
f\in\px \colon
f\begin{pmatrix}
Y & Z^1 & & \\
& \ddots & \ddots & \\
& & \ddots & Z^\ell \\
& & & Y
\end{pmatrix}=0
\quad \forall Z^1,\dots,Z^\ell
\right\}$$
for $\ell\ge0$. Then $(\cI_\ell(Y))_\ell$ is a decreasing chain of ideals in $\px$, and $\bigcap_\ell \cI_\ell(Y)=\{0\}$ by Remark \ref{r:diff}.

Two points $Y\in\mats^g$ and $Y'\in\mat{s'}^g$ are {\bf similar} if $s= s'$ and $Y'= PYP^{-1}$ for some $P\in\gl{s}$. We say that $Y\in\mats^g$ is {\bf irreducible} if $Y_1,\dots,Y_s$ do not admit a nontrivial common invariant subspace. More generally, $Y\in\mats^g$ is {\bf semisimple} if it is similar to a direct sum of irreducible points.

For $Y\in\mats^g$ let $\cS(Y)$ and $\cC(Y)$ denote the unital $\kk$-algebra in $\mats$ generated by $Y$ and the centralizer of $Y$ in $\mats$, respectively.

\begin{rem}\label{r:facts}
The following hold if $Y$ is semisimple:
\begin{enumerate}[(i)]
\item $\cC(Y)$ and $\cS(Y)$ are semisimple algebras, and the centralizer of $\cC(Y)$ in $\mats$ equals $\cS(Y)$ by the double centralizer theorem \cite[Theorem 6.2.5]{Pro1};
\item every $\cC(Y)$-bimodule homomorphism $\mats\to\mats$ is given by
$$X\mapsto \sum_{t=1}^m\widehat{a}_t X \widecheck{a}_t$$
for some $\widehat{a}_t,\widecheck{a}_t\in\cS(Y)$. This follows from (i) by a standard argument.
\end{enumerate}
\end{rem}

Finally, semisimple points $Y^1,\dots,Y^h$ are {\bf separated} if none of the irreducible blocks in $Y^i$ is similar to an irreducible block in $Y^j$, for $1\le i,j\le h$. In this case we have
$$\cS(Y^1\oplus\cdots\oplus Y^h)=\cS(Y^1)\oplus\cdots\oplus \cS(Y^h),\qquad
\cC(Y^1\oplus\cdots\oplus Y^h)=\cC(Y^1)\oplus\cdots\oplus \cC(Y^h).$$

At this point we can state the first structural result on germs about matrix points (cf. Corollaries \ref{c:compl} and \ref{c:hom} below).

\begin{prop}\label{p:prime}
If $Y$ is irreducible then $\cO_Y$ is a prime ring (and likewise for $\Oa_Y$, $\Oua_Y$ if $\kk=\C$).
\end{prop}

\begin{proof}
Let $Y\in\matn^g$ and $a=(a_\ell)_\ell,\,b=(b_\ell)_\ell\in\cO_Y\setminus\{0\}$. Then there exist minimal $m',m''\ge0$ such that $a_{m'}\neq0$ and $b_{m''}\neq0$. Since $a_{m'}$ and $b_{m''}$ are nonzero multilinear maps into $\matn$, there exists $f_0\in\matn$ such that $a_{m'} f_0 b_{m''}$ is a nonzero $(m'+m'')$-linear map. Since $Y$ is irreducible, there exists $f\in\px$ such that $f(Y)=f_0$. By the minimality of $m',m''$ we have
$$(afb)_{m'+m''}=a_{m'} f_0 b_{m''}\neq0$$
and therefore $afb\neq0$. Hence $\cO_Y$ is a prime ring. The same proof applies in the analytic case since $\pxc\subset \Oua_Y\subset\Oa_Y$.
\end{proof}

\subsection{Truncated \lac}

Next we define \lac for finite sequences of multilinear maps.

\begin{defn}\label{d:tru}
Let $Y\in\mats^g$ and $L\in\N$. A sequence $(f_\ell)_{\ell=0}^L$ of $\ell$-linear maps
$$f_\ell:\left(\mats^g\right)^\ell\to \mats$$
satisfies the {\bf truncated \lac} of order $L$ with respect to $Y$ (shortly $\LAC_L(Y)$) if for all $Z^j\in\mats^g$,
$$f_1([S,Y])=[S,f_0]$$
and
\begin{align*}
f_\ell([S,Y],Z^1,\dots, Z^{\ell-1})
& =S f_{\ell-1}(Z^1,\dots, Z^{\ell-1})-f_{\ell-1}(SZ^1,Z^2,\dots, Z^{\ell-1}), \\
f_\ell(\dots,Z^j,[S,Y], Z^{j+1},\dots)
& = f_{\ell-1}(\dots,Z^{j-1},Z^jS,Z^{j+1},\dots) \\
 &\quad -f_{\ell-1}(\dots,Z^j,SZ^{j+1},Z^{j+2},\dots), \\
f_\ell(Z^1,\dots Z^{\ell-1},[S,Y])
& =f_{\ell-1}(Z^1,\dots,Z^{\ell-2}, Z^{\ell-1} S)-f_{\ell-1}(Z^1,\dots, Z^{\ell-1})S
\end{align*}
for all $2\le\ell\le L$, $1\le j\le \ell-2$, $S\in\mats$, and
\begin{align}
 S f_L(Z^1,\dots) &= f_L(SZ^1,\dots), \\
 f_L(\dots,Z^jS,Z^{j+1},\dots) &= f_L(\dots,Z^j,SZ^{j+1},\dots), \\
f_L(\dots, Z^{\ell-1} S) &= f_L(\dots, Z^{\ell-1})S
\end{align}
for all $1\le j\le L-1$, $S\in\cC(Y)$.

We say that $(f_0)$ satisfies $\LAC_0(Y)$ if $[f_0,\cC(Y)]=0$.
\end{defn}

\begin{rem}\label{r:tru}
A sequence $(f_\ell)_{\ell=0}^\infty$ satisfies $\LAC(Y)$ if and only if $(f_\ell)_{\ell=0}^L$ satisfies $\LAC_L(Y)$ for all $L\in\N\cup\{0\}$.
\end{rem}

For $Y\in\mats$ we consider $\mats^g$ as a $\cC(Y)$-bimodule in a natural way. Since $C_1[S,Y]C_2=[C_1SC_2,Y]$ for $S\in\mats$ and $C_i\in\cC(Y)$, $[\mats,Y]$ is a sub-bimodule in $\mats^g$.

\begin{defn}\label{d:adm}
Let $Y\in\mats^g$ and $\ell\in\N$. An $\ell$-linear map $f:(\mats^g)^\ell\to\mats$ is {\bf $Y$-admissible} if it induces a $\cC(Y)$-bimodule homomorphism
$$\big(\mats^g/ [\mats,Y]\big)^{\otimes_{\cC(Y)}\ell}\to \mats.$$
\end{defn}

\begin{rem}\label{r:adm}
By comparing Definitions \ref{d:tru} and \ref{r:adm} we see that an $\ell$-linear map $f$ is $Y$-admissible if and only if $(0,\dots,0,f)$ satisfies $\LAC_\ell(Y)$.
\end{rem}

\subsection{Noncommutative algebra intermezzo}

\def\ckk{\overline{\kk}}
\def\op{{\rm op}}
\def\cCe{\cC\otimes \cC^\op}

Throughout this subsection let $\cC$ be a semisimple $\kk$-algebra. When addressing properties of $\cC$-bimodules, we can identify them as (left) $\cCe$-modules, where $\cC^\op$ is the opposite algebra of $\cC$. Here $\cC^\op$ agrees with $\cC$ as a vector space over $\kk$, and multiplication satisfies $a^\op\cdot b^\op = (b\cdot a)^\op$. Since tensor product is distributive over direct sum, the $\kk$-algebra $\cCe$ is also semisimple, so every $\cCe$-module is semisimple by \cite[Proposition 6.2.2]{Pro1}, i.e., a direct sum of simple (or irreducible) modules. Furthermore, there are only finitely many simple $\cCe$-modules up to isomorphism, say $W_1,\dots, W_d$. By Schur's lemma \cite[Theorem 6.1.7]{Pro1}, $\End_{\cCe}(W_i,W_i)$ is a finite dimensional division algebra over $\kk$, and $\Hom_{\cCe}(W_i,W_j)=\{0\}$ for $i\neq j$.

Let $U,V$ be finitely generated $\cCe$-modules. Then
$$U\cong\bigoplus_i W_i^{m_i},\qquad V\cong\bigoplus_i W_i^{n_i}$$
for some $m_i,n_i$, and $U_i=W_i^{m_i}$ and $V_i=W_i^{n_i}$ are {\bf isotypic components} of type $i$ of $U$ and $V$, respectively \cite[Subsection 6.2.3]{Pro1}. By \cite[Proposition 6.2.3.1]{Pro1} we have
\begin{equation}\label{e:hom}
\Hom_{\cCe}(U,V)= \bigoplus_i \Hom_{\cCe}(U_i,V_i) \cong \bigoplus_i \End_{\cCe}(W_i,W_i)^{n_i\times m_i}.
\end{equation}

\begin{lem}\label{l:la1}
Let $U,V$ be finitely generated $\cC$-bimodules. Let $\cT\subseteq\Hom_{\cC-\cC}(U,V)$ be a subspace such that:
\begin{enumerate}
	\item for every $0\neq u\in U$ there exists $T\in\cT$ such that $Tu\neq0$;
	\item $\Phi\circ \cT\subseteq\cT$ for every $\Phi\in\End_{\cC-\cC}(V,V)$.
\end{enumerate}
Then $\cT=\Hom_{\cC-\cC}(U,V)$.
\end{lem}

\begin{proof}
First assume that $\kk$ is algebraically closed. Then $\End_{\cCe}(W_i,W_i)=\kk$ for all $i$. By \eqref{e:hom} it suffices to show that $L\in\cT$ for every $L\in \Hom_{\cC-\cC}(U_i,V_i)= \kk^{n_i\times m_i}$ and $i=1,\dots,d$. Denote $n=\rk L$. Then there exist $u_1,\dots,u_n\in U_i$ such that $L(u_1),\dots, L(u_n)$ are linearly independent in $V_i$. Clearly $u_1,\dots,u_n$ are linearly independent. By (2) it suffices to find $T\in\cT$ such that $T(u_1),\dots, T(u_n)$ are linearly independent. To simplify the notation we without loss of generality assume $U=U_i$ and $V=V_i$ for a fixed $i$.

Suppose $T(u_1),\dots, T(u_n)$ are linearly dependent for all $T\in\cT$. For $i=1,\dots,n$ let
$$\phi_i:\cT\to V,\qquad T\mapsto T(u_i).$$
Then $\phi_1(T),\dots,\phi_n(T)$ are linearly dependent for all $T$, so by \cite[Theorem 2.2]{BS} there exist $\alpha_1,\dots,\alpha_n\in\kk$, not all $0$, such that
\begin{equation}\label{e:rk}
\rk\left(\sum_i\alpha_i\phi_i\right)\le n-1.
\end{equation}
Let $u=\sum_i\alpha_iu_i\in V$. If $u\neq0$, then for every $v\in V$ there exists $T\in\cT$ such that $Tu=v$ by (1) and (2). However, this contradicts \eqref{e:rk} since $n-1<\dim V$. Therefore $u=0$ and $u_1,\dots,u_n$ are linearly dependent, a contradiction. Hence there exists $T\in\cT$ such that $T(u_1),\dots, T(u_n)$ are linearly independent.

Finally, let $\kk$ be an arbitrary field of characteristic $0$. Suppose the conclusion of the lemma fails, i.e.,
\begin{equation}\label{e:61}
\dim_{\kk}\cT<\dim_{\kk}\Hom_{\cC-\cC}(U,V).
\end{equation}
Let $\ckk$ be the algebraic closure of $\kk$. Then the $\ckk\otimes \cC$-bimodules $\ckk\otimes U,\ckk\otimes V$ and the subspace $\ckk\otimes \cT$ satisfy the assumptions of the lemma, so 
\begin{equation}\label{e:62}
\dim_{\ckk}(\ckk\otimes\cT)
=\dim_{\ckk}\Hom_{\ckk\otimes\cC-\ckk\otimes\cC}(\ckk\otimes U,\ckk\otimes V).
\end{equation}
However, \eqref{e:61} and \eqref{e:62} contradict
$$\dim_{\ckk}(\ckk\otimes\cT)=\dim_{\kk}\cT,\qquad \dim_{\ckk}\Hom_{\ckk\otimes\cC-\ckk\otimes\cC}(\ckk\otimes U,\ckk\otimes V) = \dim_{\kk}\Hom_{\cC-\cC}(U,V).$$
\end{proof}

\begin{lem}\label{l:la2}
Let $U,V$ be finitely generated $\cC$-bimodules, and let $\cA$ be a simple algebra containing $\cC$ as a subalgebra. For every $\phi\in\Hom_{\cC-\cC}(U\otimes_{\cC} V, \cA)$ there exist $m\in\N$ and $\widehat{\phi}_t\in\Hom_{\cC-\cC}(U,\cA)$, $\widecheck{\phi}_t\in\Hom_{\cC-\cC}(V,\cA)$ for $1\le t\le m$ such that
\begin{equation}\label{e:63}
\phi(u\otimes_{\cC}v)=\sum_{t=1}^m \widehat{\phi}_t(u)\widecheck{\phi}_t(v)
\end{equation}
for all $u\in U$ and $v\in V$.
\end{lem}

\begin{proof}
By distributivity of $\otimes_{\cC}$ and $\Hom_{\cC-\cC}$ over direct sum it suffices to assume that $U$ and $V$ are simple $\cC-\cC$-bimodules. Moreover, by \cite[Corollary 6.1.9.1]{Pro1} we can further assume that $U=L_1\otimes  L_2^{\op}$ and $V=L_3\otimes  L_4^{\op}$ for some minimal left ideals $L_i\subset\cC$. By \cite[Theorem 6.3.1(2)]{Pro1} we have $L_i=\cC c_i$ for some idempotents $c_i\in\cC$. We distinguish two cases. If $c_2c_3=0$, then $L_2^\op\otimes_{\cC} L_3=\{0\}$ and $U\otimes_{\cC}V=\{0\}$, so the lemma trivially holds. Hence assume $c_2c_3\neq0$, and let
$$a=\phi\big((c_1\otimes c_2)\otimes_{\cC} (c_2\otimes c_4)\big)\in\cA.$$
Since $\cA$ is simple, there exist $\widehat{a}_t,\widecheck{a}_t\in\cA$ such that
\begin{equation}\label{e:64}
a=\sum_t \widehat{a}_tc_2c_3 \widecheck{a}_t
\end{equation}
Define $\widehat{\phi}_t\in\Hom_{\cC-\cC}(U,\cA)$ and $\widecheck{\phi}_t\in\Hom_{\cC-\cC}(V,\cA)$ by
$$\widehat{\phi}_t(c_1\otimes c_2)=c_1\widehat{a}_tc_2,\qquad \widecheck{\phi}(c_2\otimes c_4)=c_3\widecheck{a}_tc_4.$$
Since $\phi$ is a $\cC$-bimodule homomorphism and $c_i$ are idempotents, we have $c_1ac_4=a$ and thus \eqref{e:63} holds by \eqref{e:64}.
\end{proof}

In Section \ref{sec7} we will also require the following fact.

\begin{lem}\label{l:n0}
Let $\cA$ be a central simple $\kk$-algebra containing $\cC$ as a subalgebra. Then $\Hom_{\cC-\cC}(U,\cA)\neq\{0\}$ for every nonzero $\cC$-bimodule $U$.
\end{lem}

\begin{proof}
Since $\cA$ is a central simple algebra over $\kk$, we have $\cA\otimes \cA^\op\cong\End_{\kk}(\cA)$. Therefore $\cA$ is a faithful left $\cA\otimes \cA^\op$-module, i.e., for every $a\in\cA\otimes \cA^\op\setminus\{0\}$ there exists $m\in\cA$ such that $a\cdot m\neq0$. Then $\cA$ is also a faithful left $\cCe$-module. Every simple $\cCe$-module is isomorphic to a minimal left ideal in $\cCe$ by \cite[Corollary 6.1.9.1]{Pro1}. On the other hand, every minimal left ideal in $\cCe$ is isomorphic to a $\cCe$-submodule of $\cA$ since $\cA$ is faithful. Since every $\cCe$-module $U$ is a direct sum of simple modules by semisimplicity, there exists a nonzero $\cCe$-homomorphism $U\to\cA$. 
\end{proof}

\subsection{Hermite interpolation}

We prove our main interpolation result, Theorem \ref{t:hermite}, using the algebraic tools derived in the previous subsection.

\begin{lem}\label{l:sep1}
Let $Y\in\mats^g$ be a semisimple point and $Z\in\mats^g\setminus [\mats,Y]$. Then there exists $f\in\px$ such that
$$f(Y)=0,\quad f\begin{pmatrix}
Y & Z \\
0 & Y
\end{pmatrix}\neq 0.$$
\end{lem}

\begin{proof}
Suppose
$$f(Y)=0 \quad\implies\quad
f\begin{pmatrix}
Y & Z \\ 0 & Y
\end{pmatrix}=0$$
for all $f\in\px$.
Hence there is a unital homomorphism of algebras $\cS(Y)\to \mat{2s}$ determined by
$$Y_j\mapsto
\begin{pmatrix}
Y_j & Z_j \\ 0 & Y_j
\end{pmatrix}$$
for $j=1,\dots,g$. By the version of Skolem-Noether theorem in \cite[Lemma 3.10]{KV} there exists $P=(P_{ij})_{i,j=1}^2\in\gl{2s}$ such that
\begin{equation}\label{e:sn}
\begin{pmatrix}P_{11} & P_{12} \\ P_{21} & P_{22}\end{pmatrix}
\begin{pmatrix}
Y & Z \\ 0 & Y
\end{pmatrix}
=\begin{pmatrix}
Y & 0 \\ 0 & Y
\end{pmatrix}
\begin{pmatrix}P_{11} & P_{12} \\ P_{21} & P_{22}\end{pmatrix}.
\end{equation}
Therefore $[P_{i1},Y]=0$ for $i\in\{1,2\}$. Since $P$ is invertible, there exists $A\in\mats$ such that $P_{21}+AP_{11}\in \gl{s}$. Moreover, since $P_{11},P_{21}\in\cC(Y)$ and $\cC(Y)$ is semisimple, one can choose $A\in\cC(Y)$. Then we can replace $P$ with
$$\begin{pmatrix}I & 0 \\ A & I\end{pmatrix}P=
\begin{pmatrix}P_{11} & P_{12} \\ P_{21}+A P_{11} & P_{22}+A P_{12}\end{pmatrix}$$
and the relation \eqref{e:sn} still holds. So we can without loss of generality assume that $P_{21}$ is invertible. Furthermore, \eqref{e:sn} implies $P_{21}Z=[Y,P_{22}]$. Therefore $Z=[Y,P_{21}^{-1}P_{22}]$, a contradiction.
\end{proof}

\begin{prop}\label{p:sep}
Let $\ell\in\N$. For $i=1,\dots,h$ let $Y^i\in\mat{s_i}^g$ be separated semisimple points, and let $f_i:(\mats^g)^\ell\to\mats$ be $Y^i$-admissible $\ell$-linear maps. Then there exists $f\in\px$ such that
\begin{equation}\label{e:40}
f\in\cI_{\ell-1}(Y^i),\qquad \Delta^\ell_{Y^i}f=f_i
\end{equation}
for all $i=1,\dots,h$.
\end{prop}

\begin{proof}
We prove the statement by induction on $\ell$.

Let $\ell=1$. For a fixed $i$ let $U=\mat{s_i}^g/[\mat{s_i},Y^i]$, $V=\mat{s_i}$, and
$$\cT=\left\{\Delta^1_{Y^i}f\colon f\in\cI_0(Y^i) \right\}.$$
Observe that $\cT$ satisfies the hypotheses of  Lemma \ref{l:la1} for $\cC=\cC(Y^i)$. The condition (1) holds by Lemma \ref{l:sep1}. Next, every $\cC(Y^i)$-bimodule endomorphism $\Phi$ of $\mat{s_i}$ is of the form
$$\Phi:X\mapsto \sum_t\widehat{a}_t X \widecheck{a}_t$$
for some $\widehat{a}_t,\widecheck{a}_t\in\cS(Y^i)$ by Remark \ref{r:facts}(ii). There exist $\widehat{f}_t,\widecheck{f}_t\in\px$ such that $\widehat{f}_t(Y^i)=\widehat{a}_t$ and $\widecheck{f}_t(Y^i)=\widecheck{a}_t$ for all $t$. For every $f\in\cI_0(Y)$ we then have
$$(\widehat{f}_tf\widecheck{f}_t)
\begin{pmatrix}Y^i & Z \\ 0 & Y^i\end{pmatrix}
=\begin{pmatrix}0 & \widehat{a}_t(\Delta^1_{Y^i}f(Z))\widecheck{a}_t \\ 0 & 0\end{pmatrix}
$$
and thus
$$\Phi\circ \Delta^1_{Y^i}f=\Delta^1_{Y^i}\left(\sum_t \widehat{f}_tf\widecheck{f}_t \right).$$
Hence the condition (2) is satisfied, so $\cT$ is precisely the subspace of $Y^i$-admissible linear maps by Lemma \ref{l:la1}.

Therefore for each $i$ there exists $\widehat{f}_i\in\cI_0(Y^i)$ such that $f_i=\Delta^1_{Y^i}\widehat{f}_i$. Furthermore, since $Y^1,\dots,Y^h$ are separated, the algebra $\cS(Y^1\oplus\cdots\oplus Y^h)$ contains
$$ (\oplus^{i-1} 0) \oplus I \oplus (\oplus^{h-i}0)$$
for $i=1,\dots,h$. Therefore there exist $\widecheck{f}_i\in\px$ such that $\widecheck{f}_i(Y^{i'})=\delta_{ii'}I$, where $\delta_{ii'}$ is the Kronecker delta. Then $f=\sum_i\widehat{f}_i\widecheck{f}_i$ satisfies \eqref{e:40}, so the basis of induction is proven.

Now let $\ell\ge2$ and assume the statement holds for $\ell-1$. By Lemma \ref{l:la2} and Definition \ref{d:adm} there exist $Y^i$-admissible linear maps $\widehat{f}_{it}:\mat{s_i}^g\to \mat{s_i}$ and $Y^i$-admissible $(\ell-1)$-linear maps $\widecheck{f}_{it}:(\mat{s_i}^g)^{\ell-1}\to \mat{s_i}$ such that
$$f_i(Z^1,\dots,Z^\ell)=\sum_t \widehat{f}_{it}(Z^1)\widecheck{f}_{it}(Z^2,\dots,Z^\ell)$$
for all $1\le i\le h$. By the basis of induction and the induction hypothesis there exist $\widehat{f}_t\in\bigcap_i\cI_0(Y^i)$ and $\widecheck{f}_t\in\bigcap_i\cI_{\ell-2}(Y^i)$ such that
$$\Delta^1_{Y^i}\widehat{f}_t=\widehat{f}_{it},\qquad
\Delta^{\ell-1}_{Y^i}\widecheck{f}_t=\widecheck{f}_{it}$$
for all $i$. Then
$$(\widehat{f}_t\widecheck{f}_t)\begin{pmatrix}
Y^i & Z^1 & 0 & \cdots & 0\\
0 & Y^i & Z^2 & & \vdots\\
\vdots& & \ddots & \ddots & 0 \\ 
\vdots& & & \ddots & Z^\ell \\
0 & \cdots& \cdots & 0 & Y^i
\end{pmatrix}$$
equals
\begin{align*}
& \begin{pmatrix}
0 & \widehat{f}_{it}(Z^1) & * & \cdots & *\\
\vdots & 0 & * & \cdots& \vdots\\ 
\vdots & & \ddots & \ddots & \vdots \\
\vdots & & & \ddots & * \\
0 & \cdots & 0 & \cdots & 0
\end{pmatrix}
\begin{pmatrix}
* & \cdots & \cdots & \cdots & *\\
0 & 0 & \cdots & 0 & \widecheck{f}_{it}(Z^2,\dots,Z^\ell)\\ 
\vdots & & \ddots & & 0 \\
\vdots & & & \ddots & \vdots \\
0 & \cdots & \cdots & \cdots & 0
\end{pmatrix} \\
=& \begin{pmatrix}
0 & \cdots & \cdots & 0 & \widehat{f}_{it}(Z^1)\widecheck{f}_{it}(Z^2,\dots,Z^\ell)\\
\vdots & \ddots & & & 0\\ 
\vdots & & & & \vdots \\
\vdots & & & \ddots & \vdots \\
0 & \cdots & \cdots & \cdots & 0
\end{pmatrix}
\end{align*}
by Remark \ref{r:diff}. Therefore $f=\sum_t \widehat{f}_t\widecheck{f}_t\in\px$ satisfies \eqref{e:40}.
\end{proof}

\begin{thm}\label{t:hermite}
For $i=1,\dots,h$ let $Y^i\in\mat{s_i}^g$ be separated semisimple points, and $L\in\N\cup\{0\}$. If $(f^{(i)}_\ell)_{\ell=0}^L$ are sequences of multilinear maps satisfying $\LAC_L(Y^i)$, then there exists $f\in\px$ such that
\begin{equation}\label{e:hermite}
\Delta^\ell_{Y^i}f=f^{(i)}_\ell
\end{equation}
for all $1\le i\le h$ and $0\le \ell\le L$.
\end{thm}

\begin{proof}
We prove the statement by induction on $L$. The basis of induction $L=0$ holds because $[f_0^{(i)},\cC(Y^i)]=0$ if and only if $f_0^{(i)}\in\cS(Y^i)$ by Remark \ref{r:facts}(i). Now assume that the statement holds for $L-1$. Then there exists $\widehat{f}\in\px$ such that \eqref{e:hermite} holds for all $1\le i\le h$ and $0\le \ell\le L-1$. Let $h^{(i)}_L:=f^{(i)}_L-\Delta^L_{Y^i}\widehat{f}$. Then $h^{(i)}_L$ is a $Y^i$-admissible $L$-linear map by Remark \ref{r:adm}. By Proposition \ref{p:sep} there exists $\widecheck{f}\in\bigcap_i \cI_{L-1}(Y^i)$ such that $\Delta^L_{Y^i}\widecheck{f}=h^{(i)}_L$ for all $i$. Then $f=\widehat{f}+\widecheck{f}$ satisfies \eqref{e:hermite} for $L$.
\end{proof}

\begin{exa}
Let $Y=(E_{12},E_{21})\in\mat{2}^2$. Then $Y$ is an irreducible point and $r(x_1,x_2)=[x_1,x_2]^{-1}\in\cO_Y$. A direct computation shows that
$$f=3(x_1x_2-x_2x_1)+2(x_2x_1x_2x_1-x_1x_2x_1x_2)$$
satisfies $\Delta^\ell_Yf=\Delta^\ell_Yr$ for $\ell\in\{0,1\}$. By brute force one can also check that a minimal-degree polynomial $f$ satisfying $\Delta^\ell_Yf=\Delta^\ell_Yr$ for $\ell\in\{0,1,2\}$ has degree 8.
\end{exa}

\begin{rem}\label{r:bounds}
One can also derive polynomial bounds on the degree of $f$ as in Theorem \ref{t:hermite}. By Remark \ref{r:diff}, the maps $f_\ell=\Delta_Y^\ell f$ for $\ell\le L$ are determined by the action of $f$ on the $Lg s_i^2$ tuples
$$\begin{pmatrix}
	Y^i & Z^1 &  & \\
	& \ddots & \ddots & \\
	& & \ddots & Z^L \\
	& & & Y^i
\end{pmatrix}\in\mat{(L+1)s_i}^g$$
for all $Z^1\otimes\cdots\otimes Z^L$ in some basis for $(\mat{s_i}^g)^{\otimes L}$ and $1\le i\le h$. Write
$$N=L(L+1)g\sum_{i=1}^hs_i^3$$
and let $T\in\mat{N}^g$ be the direct sum of these tuples. Let $\cA\subset\mat{N}$ be the algebra generated by $T_1,\dots,T_g$. By \cite[Theorem 3]{Shi}, $\cA$ is generated by polynomials in $T$ of degree $2N\log_2 N+4N-4$. Therefore there exists $\widetilde{f}\in\px$ of degree at most $2N\log_2 N+4N-4$ such that $f_\ell=\Delta_Y^\ell f=\Delta_Y^\ell \widetilde{f}$ for $\ell\le L$.
\end{rem}

\begin{rem}
The analog of Theorem \ref{t:hermite} fails for non-semisimple points by \cite[Example 3.10]{AM1}. Also, on first glance one might think that for proving Theorem \ref{t:hermite}, it suffices to first show a simpler version of it for collections of irreducible points, in which case the bimodule formalism is redundant. But this is not true since an nc function about a semisimple point $Y$ is not determined by its ``restrictions'' to irreducible blocks of $Y$; see the next remark for a rigorous statement.
\end{rem}

\begin{rem}\label{r:ss}
For arbitrary points $Y'$ and $Y''$ there is a canonical $\kk$-algebra homomorphism
\begin{equation}\label{e:can}
\cO_{Y'\oplus Y''}\to \cO_{Y'}\times \cO_{Y''}.
\end{equation}
Indeed, in Subsection \ref{ssec2.4} we saw that every formal nc germ about $Y'\oplus Y''$ determines an nc function on $\Nilp(Y'\oplus Y'')$. Since nc functions respect direct sums and similarities, it is easy to see that for all
$$X\in \Nilp(Y'\oplus Y'')\cap \big(\matn \otimes (\mat{s'}\oplus \mat{s''})\big)^g$$
we have
$$f(X)\in \matn \otimes (\mat{s'}\oplus \mat{s''}).$$
Consequently, if $K$ is the permutation matrix corresponding to the canonical shuffle of blocks
$$\mat{ns'}\oplus \mat{ns''}\to \matn \otimes (\mat{s'}\oplus \mat{s''}),$$
then for all $X'\oplus X''\in \Nilp(Y'\oplus Y'')$,
\begin{equation}\label{e:dir}
f(K(X'\oplus X'')K^{-1})=K(f'(X')\oplus f''(X''))K^{-1}
\end{equation}
for some nc functions $f'$ and $f''$ on $\Nilp(Y')$ and $\Nilp(Y'')$, respectively. Thus \eqref{e:can} is given by $f\mapsto (f',f'')$. If $\kk=\C$ and $f$ is (uniformly) analytic, then $f'$ and $f''$ are also (uniformly) analytic by \eqref{e:dir}. Thus the homomorphism \eqref{e:can} restricts to homomorphisms
\begin{equation}\label{e:can1}
\Oa_{Y'\oplus Y''}\to \Oa_{Y'}\times \Oa_{Y''},\qquad
\Oua_{Y'\oplus Y''}\to \Oua_{Y'}\times \Oua_{Y''}.
\end{equation}
We refer to \cite[Chapter 9]{KVV3} for further discussion. Corollary \ref{c:inj} below demonstrates that homomorphisms \eqref{e:can} and \eqref{e:can1} are not necessarily injective.
\end{rem}

\subsection{Completions of the free algebra}

In this subsection we apply Hermite interpolation for nc functions to investigate the ring structure of nc germs about semisimple points.

\begin{prop}\label{p:power}
If $Y$ is a semisimple point, then $\cI_\ell(Y)=\cI_0(Y)^{\ell+1}$ for all $\ell\in\N$.
\end{prop}

\begin{proof}
By Proposition \ref{p:sep}, there is a natural isomorphism between $\cI_\ell(Y)/\cI_{\ell+1}(Y)$ and all $Y$-admissible $(\ell+1)$-linear maps. Therefore
$$\cI_\ell(Y)\equiv\cI_0(Y)\cdot\cI_{\ell-1}(Y) \mod \cI_{\ell+1}(Y)$$
follows by Lemma \ref{l:la2} as in the last part of the proof of Proposition \ref{p:sep}. Furthermore, $\cI_0(Y)\cI_{\ell-1}(Y)\subseteq\cI_\ell (Y)\subseteq \cI_{\ell-1}(Y)$, so
\begin{align*}
\cI_{\ell-1}(Y)/\cI_\ell (Y)
&= (\cI_{\ell-1}(Y)/\cI_{\ell+1}(Y)) \big/ (\cI_\ell(Y)/\cI_{\ell+1}(Y)) \\
&= (\cI_{\ell-1}(Y)/\cI_{\ell+1}(Y)) \big/ ((\cI_0(Y)\cI_{\ell-1}(Y))/\cI_{\ell+1}(Y)) \\
&= \cI_{\ell-1}(Y)/(\cI_0(Y)\cI_{\ell-1}(Y)),
\end{align*}
implies $\cI_\ell(Y)=\cI_0(Y)\cdot\cI_{\ell-1}(Y)$.
\end{proof}

\begin{cor}\label{c:compl}
If $Y$ is a semisimple point, then
$$\cO_Y = \varprojlim_\ell \big(\px/\cI_0(Y)^\ell\big).$$
\end{cor}

\begin{proof}	
Interpolating polynomials of Theorem \ref{t:hermite}, together with Proposition \ref{p:power}, induce surjective homomorphisms $\cO_Y\to \px/\cI_0(Y)^\ell$ such that the diagram
\begin{center}\begin{tikzpicture}[node distance=1.6cm, auto]
\node (O) {$\cO_Y$};
\node (P) [above of=O] {$\px$};
\node (Fl) [below of=O, left of=O] {};
\node (Fm) [below of=O, right of=O] {};
\node (Pl) [left of=Fl] {$\px/\cI_0(Y)^\ell$};
\node (Pm) [right of=Fm] {$\px/\cI_0(Y)^m$};
\draw[->] (P.270) to node {} (O.90);
\draw[->] (O.210) to node {} (Pl.30);
\draw[->] (O.330) to node {} (Pm.150);
\draw[->] (P.220) to node {} (Pl.120);
\draw[->] (P.320) to node {} (Pm.60);
\draw[->] (Pl) to node {} (Pm);
\end{tikzpicture}\end{center}
commutes for all $\ell>m$. Hence there is a surjective homomorphism
\begin{equation}\label{e:iso}
\cO_Y\to \varprojlim_\ell \big(\px/\cI_0(Y)^\ell\big).
\end{equation}
Furthermore, if $f\in\cO_Y$ is nonzero, then there exists $\ell$ such that $\Delta_{Y}^\ell f\neq0$, so the image of $f$ in $\px/\cI_0(Y)^\ell$ is nonzero. Hence \eqref{e:iso} is an isomorphism.
\end{proof}

We continue by noting some apparent isomorphisms of formal germ algebras.

\begin{lem}\label{l:ss}
If $Y\in\mats^g$ and $P\in\gl{s}$, then $\cO_{PYP^{-1}}\cong \cO_Y$. Furthermore, for arbitrary $Y^1,\dots,Y^h\in\kk_{\nc}^g$ and $m_1,\dots,m_h\in\N$ we have
\begin{equation}\label{e:65}
\cO_{\bigoplus_i (\oplus^{m_i}Y^i)}\cong\cO_{\bigoplus_i Y^i}.
\end{equation}
\end{lem}

\begin{proof}
The first claim is obvious. Now let $m=\max\{m_1,\dots,m_h\}$. As in Remark \ref{r:ss}, there are canonical homomorphisms
$$\phi: \cO_{\oplus^m(\bigoplus_i Y^i)}\to \cO_{\bigoplus_i (\oplus^{m_i}Y^i)},\qquad
\psi: \cO_{\bigoplus_i (\oplus^{m_i}Y^i)}\to \cO_{\bigoplus_i Y^i}.$$
Their composition $\psi\circ\phi$ is again a canonical homomorphism of the same kind, and is an isomorphism by \eqref{e:amp}. By the construction of $\phi,\psi$ as in Remark \ref{r:ss} it is also straightforward to see that $\phi(\psi\circ\phi)^{-1}\psi$ is the identity map, so $\psi$ is an isomorphism.
\end{proof}

The following theorem greatly generalizes the observation $\cO_Y\cong \cO_0$ for $Y\in\kk^g$ used in Section \ref{sec5}, and classifies $\cO_Y$ in terms of $Y$. See also \cite{SSS} for results about correspondences between noncommutative varieties and algebras of nc functions on them.

\begin{thm}\label{t:iso}
Let $Y$ and $Y'$ be semisimple points. Then the rings $\cO_Y$ and $\cO_{Y'}$ are isomorphic if and only if $\cS(Y)\cong\cS(Y')$.

The same conclusion holds for (uniformly) analytic nc germs about $Y$ and $Y'$ if $\kk=\C$.
\end{thm}

\begin{proof}
$(\Rightarrow)$ The description of $\cO_Y$ given by Corollary \ref{c:compl} implies that $\cO_Y$ admits $h$ maximal ideals, where $h$ is the number of simple factors in $\cS(Y)$, and their intersection equals $\cI_0(Y)$. Thus an isomorphism $\cO_Y\to\cO_{Y'}$ maps $\cI_0(Y)$ to $\cI_0(Y')$, and so it induces an isomorphism
$$\cS(Y)\cong \px/\cI_0(Y) \to \px/\cI_0(Y')\cong \cS(Y').$$

$(\Leftarrow)$ By Lemma \ref{l:ss} it suffices to assume that $Y,Y'\in\mats^g$ are direct sums of pairwise non-similar irreducible points. Moreover, since $\cS(Y)\cong\cS(Y')$, each irreducible block of $Y$ is similar to an irreducible block of $Y'$, we can further replace $Y'$ by a similar matrix point to obtain $\cS(Y)=\cS(Y')$. Then also $\cC(Y)=\cC(Y')$, so there is a $\cC(Y)$-isomorphism
\begin{equation}\label{e:66}
[\mats,Y]\to[\mats,Y'],\qquad [S,Y]\mapsto [S,Y'].
\end{equation}
Since $\mats^g$ is a semisimple $\cC(Y)$-bimodule, the isomorphism \eqref{e:66} extends to a $\cC(Y)$-bimodule isomorphism $L:\mats^g\to \mats^g$. Write $L=(L_1,\dots,L_g)$ for $L_j:\mats^g\to\mats$. Then $(Y'_j,L_j)$ satisfy $\LAC_1(Y)$ for all $j$, so there exist $F_1,\dots,F_g\in\px$ such that
$$F_j(Y)=Y'_j,\qquad \Delta_Y^1 F_j=L_j.$$
Since $L$ is an isomorphism, the nc polynomial map $F=(F_1,\dots,F_g)$ admits an inverse nc map $G=(G_1,\dots,G_g)$ about $Y'$ by the inverse function theorem for nc functions \cite[Theorem 1.7]{AKV}, which is uniformly analytic if $\kk=\C$ by \cite[Theorem 1.4]{AKV}. Also note that $G_j\in\cO_{Y'}$. By Corollary \ref{c:compl}, the homomorphisms
\begin{alignat*}{2}
\phi&:\px\to\px, \qquad & x &\mapsto G,\\
\psi&:\px\to \cO_Y,\qquad & x &\mapsto F
\end{alignat*}
extend to homomorphisms
$$\Phi:\cO_Y\to\cO_{Y'},\qquad \Psi:\cO_{Y'}\to \cO_Y.$$
Since $F$ and $G$ are inverse maps, $\Phi$ and $\Psi$ are inverse homomorphisms.
\end{proof}

\begin{rem}
In the proof of Theorem \ref{t:iso} we saw that for any two irreducible points $Y,Y'\in\mats^g$, there exist an nc polynomial map $F$ and a uniformly analytic nc map $G$ on a neighborhood of $Y'$ such that $F(Y)=Y'$, $G(Y')=Y$ and $F\circ G = G\circ F =\id$. It is natural to ask whether we can choose $F,G$ in such a way that $G$ is also polynomial, that is, whether we can find an nc polynomial automorphism $F$ of the noncommutative space $\kk^g_{\nc}$ such that $F(Y)=Y'$.

The answer is positive if $g\ge s+1$ or if $Y$ and $Y'$ are {\bf saturated} (meaning that $Y$ and $Y'$ without last components are already irreducible points) by \cite[Theorems 4.3 and 4.4]{Rei}. However, in general there might not be any nc polynomial automorphism $F$ mapping $Y$ to $Y'$. For example, let $g=2$. To a point $Y=(Y_1,Y_2)$ we assign the span of its commutator $L_Y=\kk\cdot [Y_1,Y_2]\subset\mats$. By \cite[Theorem]{Dic}, every nc polynomial automorphism $F$ preserves $L_Y$. On the other hand, there clearly exist irreducible points $Y,Y'\in\mats^2$ such that $L_Y\neq L_{Y'}$ if $s\ge2$.

Polynomial automorphisms of the noncommutative space are well-understood through the solution of the free Jacobian conjecture and the free Grothendieck theorem \cite{Pas,Aug}.
\end{rem}

\section{Minimal propagation and nilpotent noncommutative functions}\label{sec7}

In this section we describe a particular propagation of a sequence satisfying truncated \lac about a semisimple point, into a uniformly analytic function, which is quite distinct from the Hermite interpolation with nc polynomials described earlier. We construct an embedding of $\cS(Y)$ into $\Oua_Y$, and thus the first example of a nilpotent uniformly analytic nc function. Lastly, we also provide an nc function that vanishes on uniformly open neighborhoods of $Y'$ and $Y''$ but not of $Y'\oplus Y''$.

Before an auxiliary lemma we observe the following. Let $D:\R[t]\to\R[t]$ be the linear map $D(p)=(p-p(0))/t$. For $\ell,m>0$ we have
\begin{equation}\label{e:bin1}
D^{m-1}\big((t+1)^{\ell-1}\big)
+D^m\big((t+1)^{\ell-1}\big)=D^m\big((t+1)^\ell\big)
\end{equation}
and
\begin{equation}\label{e:bin2}
D^m\big((t+1)^\ell\big)-t D^{m+1}\big((t+1)^\ell\big)=\binom{\ell}{m}
\end{equation}
by the binomial coefficient formulas.

\begin{lem}\label{l:seq}
Let $\alpha,\beta>0$. For $\ell\in\N\cup\{0\}$ and $-1\le m\le \ell$ let $c_{\ell,m}\in\R_{\ge0}$ satisfy
\begin{align*}
c_{0,0} &= 1, \\
c_{\ell,\ell} &= 0\quad \text{and} \quad  c_{\ell,-1} = c_{\ell,0} \qquad \text{for } \ell>0, \\
c_{\ell,m} &\le \beta \max\left\{
c_{\ell,m+1}, \alpha(c_{\ell-1,m-1}+c_{\ell-1,m})
\right\}  \qquad \text{for } -1<m<\ell.
\end{align*}
Then
$$\limsup_{\ell\to\infty}\sqrt[\ell]{c_{\ell,0}}<\infty.$$
\end{lem}

\begin{proof}
It suffices to assume
$$c_{\ell,m} = \beta \max\left\{
c_{\ell,m+1}, \alpha(c_{\ell-1,m-1}+c_{\ell-1,m})
\right\}  \qquad \text{for } -1<m<\ell$$
and $\beta\ge2$. First we compute $c_{1,0}=2\alpha\beta$. Then we claim that
\begin{align}\label{e:bd1}
c_{\ell,m} &= 2\alpha^\ell \beta^\ell D^{m-1}\left((t+1)^{\ell-2}\right)|_{t=\beta} \quad \text{for } 0<m, \\ \label{e:bd2}
c_{\ell,0} &=2\alpha^\ell \beta^{\ell+1}(\beta+1)^{\ell-2}
\end{align}
for $\ell\ge2$. First observe that \eqref{e:bd2} follows from \eqref{e:bd1} since
$$c_{\ell,0}= \beta \max\left\{ c_{\ell,1}, 2\alpha c_{\ell-1,0}\right\}
=\beta \max\left\{2\alpha^\ell \beta^\ell(\beta+1)^{\ell-2}, 2\alpha \cdot 2\alpha^{\ell-1} \beta^\ell(\beta+1)^{\ell-3}\right\}
$$
and $\beta\ge1$ if $\ell>2$, and
$$c_{2,0}= \beta \max\left\{ c_{2,1}, 2\alpha c_{1,0}\right\}
=\beta \max\left\{2\alpha^2 \beta^2, 2\alpha \cdot 2\alpha\beta\right\}$$
since $\beta\ge2$. Moreover, \eqref{e:bd1} clearly holds for $m=\ell$. Next we prove \eqref{e:bd1} by increasing induction on $\ell$ and decreasing induction on $m$. By definition we have
$$c_{2,1} = \beta\alpha(c_{1,0}+c_{1,1})=2\alpha^2\beta^2,$$
so \eqref{e:bd1} holds for $\ell=2$. Now let $2<\ell$ and $1<m<\ell$. By the induction hypothesis we have
\begin{align*}
c_{\ell,m} &= \beta \max\left\{
c_{\ell,m+1}, \alpha(c_{\ell-1,m-1}+c_{\ell-1,m})
\right\} \\
&= 2\alpha^\ell\beta^\ell \max\left\{
\beta D^m\left((t+1)^{\ell-2}\right)|_{t=\beta},
D^{m-2}\left((t+1)^{\ell-3}\right)|_{t=\beta}
+D^{m-1}\left((t+1)^{\ell-3}\right)|_{t=\beta}
\right\} \\
&= 2\alpha^\ell\beta^\ell \max\left\{
\beta D^m\left((t+1)^{\ell-2}\right)|_{t=\beta},
D^{m-1}\left((t+1)^{\ell-2}\right)|_{t=\beta}
\right\} \\
&= 2\alpha^\ell\beta^\ell D^{m-1}\left((t+1)^{\ell-2}\right)|_{t=\beta}
\end{align*}
by \eqref{e:bin1} and \eqref{e:bin2}. Furthermore,
\begin{align*}
c_{\ell,1} &= \beta \max\left\{
c_{\ell,2}, \alpha(c_{\ell-1,0}+c_{\ell-1,1})
\right\} \\
&= 2\alpha^\ell\beta^\ell\max\left\{
D\left((t+1)^{\ell-2}\right)|_{t=\beta},
\beta(\beta+1)^{\ell-3}+
(\beta+1)^{\ell-3}
\right\} \\
&= 2\alpha^\ell\beta^\ell\max\left\{
\left((\beta+1)^{\ell-2}-1\right)/\beta,
(\beta+1)^{\ell-2}
\right\} \\
&= 2\alpha^\ell\beta^\ell(\beta+1)^{\ell-2},
\end{align*}
so \eqref{e:bd1} holds.
\end{proof}

Let $Y\in\mats^g$ be a semisimple point. Recall that $\cC(Y)$-bimodules are semisimple, and that $\mats^g$ and $[\mats,Y]$ are $\cC(Y)$-bimodules in a natural way. Hence there exists a $\cC(Y)$-bimodule projection $\pi:\mats^g\to [\mats,Y]$.

\begin{thm}\label{t:min}
Let $Y\in\mats^g$ be a semisimple point and let $(f_\ell)_{\ell=0}^L$ satisfy $\LAC_L(Y)$ for some $L\in\N\cup\{0\}$. Then there exists a unique propagation $(f_\ell)_{\ell=0}^\infty$ satisfying $\LAC(Y)$ and
\begin{equation}\label{e:min}
f_\ell|_{(\ker\pi)^\ell}=0
\end{equation}
for $\ell>L$.

Moreover, $\limsup_{\ell\to\infty}\sqrt[\ell]{\|f_\ell\|_{\cb}}<\infty$ if $\kk=\C$.
\end{thm}

\begin{proof}
Since $\mats^g$ is a direct sum of $[\mats,Y]$ and $\ker\pi$, uniqueness follows from the definition of $\LAC(Y)$.

Since $\mats$ is a semisimple $\cC(Y)$-bimodule and the map $\mats\to[\mats,Y]$ given by $S\mapsto[S,Y]$ is a surjective $\cC(Y)$-bimodule homomorphism, it admits a $\cC(Y)$-bimodule right inverse $\phi:[\mats,Y]\to \mats$. Moreover,
\begin{equation}\label{e:0}
\phi([S,Y])-S\in \cC(Y)
\end{equation}
holds for every $S\in\mats$. Let $\sigma:\mats^g\to\ker\pi$ be the projection onto $\ker\pi$ along $[\mats,Y]$, so $Z=\pi(Z)+\sigma(Z)$ for all $Z\in\mats^g$. For $\ell>L$ and $0\le m\le \ell$ we recursively define $\ell$-linear maps
$$f_{\ell,m}:(\ker\pi)^m\times (\mats^g)^{\ell-m}\to\mats$$
by $f_{\ell,\ell} :=0$ and
\begin{equation}\label{e:ugly}
\begin{split}
f_{\ell,\ell-1}(\dots,W^{\ell-1},Z^\ell)
&:= f_{\ell-1,\ell-2}(\dots,W^{\ell-1}(\phi\circ\pi)(Z^\ell)),\\
f_{\ell,m}(\dots,W^m,Z^{m+1},\dots)
&:= f_{\ell,m+1}(\dots,W^m,\sigma(Z^{m+1}),Z^{m+2}\dots) \\
& \quad +f_{\ell-1,m-1}(\dots,W^m(\phi\circ\pi)(Z^{m+1}),Z^{m+2},\dots) \\
& \quad -f_{\ell-1,m}(\dots,W^m,(\phi\circ\pi)(Z^{m+1})Z^{m+2},\dots), \\
f_{\ell,0}(Z^1,\dots)
&:=f_{\ell,1}(\sigma (Z^1),Z^2,\dots) \\
& \quad +(\phi\circ\pi)(Z^1)f_{\ell-1,0}(Z^2,\dots) \\
& \quad -f_{\ell-1,0}((\phi\circ\pi)(Z^1)Z^2,\dots)
\end{split}
\end{equation}
for $0<m<\ell-1$. Now $f_\ell:=f_{\ell,0}$ clearly satisfy \eqref{e:min}. Next, we check $\LAC(Y)$ for $f_\ell$ by induction on $\ell$. Firstly,
\begin{align*}
f_\ell([S,Y],Z^1,\dots)
& =\phi([S,Y])f_{\ell-1}(Z^2,\dots) -f_{\ell-1}(\phi([S,Y])Z^2,\dots) \\
& =S f_{\ell-1}(Z^1,\dots)-f_{\ell-1}(SZ^1,Z^2,\dots)
\end{align*}
by the induction hypothesis and \eqref{e:0}. Next, denote $S'=(\phi\circ\pi)(Z^1)$. Then
\begingroup
\allowdisplaybreaks
\begin{align*}
f_\ell(Z^1,[S,Y], Z^2,\dots)
&=f_{\ell,1}(\sigma (Z^1),[S,Y], Z^2,\dots) \\
& \quad +(S'f_{\ell-1}([S,Y],Z^2,\dots)-f_{\ell-1}(S'[S,Y],Z^2\dots) \\
&=f_{\ell-1,0}(\sigma (Z^1)\phi([S,Y]), Z^2,\dots)
-f_{\ell-1,1}(\sigma (Z^1),\phi([S,Y])Z^2,\dots)\\
& \quad +(S'\left(\phi([S,Y])f_{\ell-2}(Z^2,\dots)
-f_{\ell-2}(\phi([S,Y])Z^2,\dots)\right) \\
& \quad -f_{\ell-1}([S'S,Y]-[S',Y]S,Z^2\dots) \\
&=f_{\ell-1}(\sigma (Z^1)S, Z^2,\dots)-f_{\ell-1}(\sigma (Z^1),SZ^2,\dots)\\
& \quad +S'\left(Sf_{\ell-2}(Z^2,\dots)-f_{\ell-2}(SZ^2,\dots)\right) \\
& \quad -S'Sf_{\ell-2}(Z^2\dots)+f_{\ell-2}(S'SZ^2\dots)+f_{\ell-1}([S',Y]S,Z^2\dots) \\
&=f_{\ell-1}(\sigma (Z^1)S, Z^2,\dots)-f_{\ell-1}(\sigma (Z^1),SZ^2,\dots)\\
& \quad -S'f_{\ell-2}(SZ^2,\dots)+f_{\ell-2}(S'SZ^2\dots)+f_{\ell-1}([S',Y]S,Z^2\dots) \\
&=f_{\ell-1}(\sigma (Z^1)S, Z^2,\dots)-f_{\ell-1}(\sigma (Z^1),SZ^2,\dots)\\
& \quad -f_{\ell-1}([S',Y],SZ^2,\dots)+f_{\ell-1}([S',Y]S,Z^2\dots) \\
&=f_{\ell-1,0}(\sigma (Z^1)S, Z^2,\dots)-f_{\ell-1,1}(\sigma (Z^1),SZ^2,\dots)\\
& \quad -f_{\ell-1}(\pi(Z^1),SZ^2,\dots)+f_{\ell-1}(\pi(Z^1)S,Z^2\dots) \\
& = f_{\ell-1}(Z^1S,Z^2,\dots)-f_{\ell-1}(Z^1,SZ^2,\dots)
\end{align*}
\endgroup
holds by \eqref{e:0} and the induction hypothesis. The rest of $\LAC(Y)$ is verified analogously.

Now let $\kk=\C$. Since $\phi$ is a linear map between finite-dimensional operator spaces, it is completely bounded; let $\alpha=\|\phi\|_{\cb}$. Similarly, let $\beta=\|\psi\|_{\cb}$, where 
$$\psi:\mats^g\to [\mats,Y]\times \ker\pi, \qquad Z\mapsto (\pi(Z),\sigma(Z)).$$
Here $[\mats,Y]\times \ker\pi$ is viewed as the $\ell_1$-direct sum of operator spaces $[\mats,Y]$ and $\ker\pi$ \cite[Section 2.6]{Pis}. Given $\ve_1,\ve_2\in\C$, the map
$$\chi_{\ve_1,\ve_2}: [\mats,Y]\times \ker\pi\to [\mats,Y]\times \ker\pi,
\qquad (X,W)\mapsto (\ve_1X,\ve_2W)$$
satisfies $\|(\chi_{\ve_1,\ve_2})\|_{\cb}=\max\{|\ve_1|,|\ve_2|\}$. By looking at $\chi_{\ve_1,\ve_2}\circ\psi$ we thus obtain
\begin{equation}\label{e:l1}
\ve_1\|\pi(Z)\|_{ns}+\ve_2\|\sigma(Z)\|_{ns}\le \beta\max\{\ve_1,\ve_2\}\|Z\|_{ns}
\end{equation}
for all $\ve_1,\ve_2\ge0$, $Z\in\matc{ns}^g$ and $n\in\N$. By \eqref{e:ugly} we have 
\begin{align*}
\|f_{\ell,m}(\dots,W^m,Z^{m+1},\dots)\|
&\le \|f_{\ell,m+1}\|\cdots\|W^m\|\|\sigma(Z^{m+1})\|\cdots\|Z^\ell\| \\
& \quad +\|f_{\ell-1,m-1}\|\cdots\|W^m(\phi\circ\pi)(Z^{m+1})\|\cdots\|Z^\ell\| \\
& \quad +\|f_{\ell-1,m}\|\cdots\|W^m\|\|(\phi\circ\pi)(Z^{m+1})Z^{m+2}\|\cdots\|Z^\ell\| \\
&\le \|f_{\ell,m+1}\|\cdots\|W^m\|\|\sigma(Z^{m+1})\|\cdots\|Z^\ell\| \\
& \quad +\alpha\|f_{\ell-1,m-1}\|\cdots\|W^m\|\pi(Z^{m+1})\|\cdots\|Z^\ell\| \\
& \quad +\alpha\|f_{\ell-1,m}\|\cdots\|W^m\|\|\pi(Z^{m+1})\|\cdots\|Z^\ell\| \\
&=\|f_{\ell,m+1}\|\|W^1\|\cdots\|Z^\ell\|\cdot \|\sigma(Z^{m+1})\| \\
& \quad +\alpha(\|f_{\ell-1,m-1}\|+\|f_{\ell-1,m}\|)
\|W^1\|\cdots\|Z^\ell\| \cdot \|\pi(Z^{m+1})\|
\end{align*}
$\ell>L$ and $m<\ell$, and thus
$$\|f_{\ell,m}(W^1,\dots,Z^\ell)\|\le
\beta\max\left\{\|f_{\ell,m+1}\|,
\alpha(\|f_{\ell-1,m-1}\|+\|f_{\ell-1,m}\|)
\right\}
\|W^1\|\cdots\|Z^\ell\|$$
by \eqref{e:l1}. Actually, the same conclusion holds for all ampliations of $f_{\ell,m}$. Therefore
$$\|f_{\ell,m}\|_{\cb}\le
\beta\max\left\{\|f_{\ell,m+1}\|_{\cb},
\alpha(\|f_{\ell-1,m-1}\|_{\cb}+\|f_{\ell-1,m}\|)_{\cb}
\right\}.$$
If $f_L=0$, then we have $f_\ell=0$ for all $\ell>L$. Otherwise, $c_{\ell,m}:=\|f_{\ell+L,m}\|_{\cb}/\|f_L\|_{\cb}$ satisfies the assumptions of Lemma \ref{l:seq}, so $\limsup_{\ell\to\infty}\sqrt[\ell]{\|f_\ell\|_{\cb}}<\infty$.
\end{proof}

By applying the ``minimal'' propagation (depending on the choice of projection $\pi$) of Theorem \ref{t:min} to one-term sequences (i.e., matrices), we obtain the following.

\begin{thm}\label{t:alg}
Let $Y\in\mats^g$ be a semisimple point and $a^1,\dots,a^t\in\cS(Y)$. Then there exist nc functions $f^1,\dots,f^t$, uniformly analytic on an nc ball about $Y$ if $\kk=\C$, such that
\begin{enumerate}[(i)]
	\item $f^i(Y)=a^i$ for $i=1,\dots,t$;
	\item for every $p\in\kk\!\Langle y_1,\dots,y_t\Rangle$,
	$$p(a^1,\dots,a^t)=0 \implies p(f^1,\dots,f^t)=0.$$
\end{enumerate}
\end{thm}

\begin{proof}
Since $[a^i,\cC(Y)]=\{0\}$, the one-term sequences $(a^1),\dots,(a^t)$ satisfy $\LAC_0(Y)$, so there exist $(f^i_\ell)_\ell$ for $i=1,\dots,t$ as in Theorem \ref{t:min}. Therefore there are uniformly analytic nc functions $f^1,\dots,f^t$ on an nc ball about $Y$ given by
$$f^i(X)=\sum_{\ell=0}^\infty f^i_\ell(X-\oplus^nY,\dots,X-\oplus^nY)$$
by Theorem \ref{t:old4}. In particular, they satisfy $f^i(Y)=a^i$ and $f^i_\ell|_{(\ker\pi)^\ell}=0$ for $\ell>0$.

Suppose $p(a^1,\dots,a^t)=0$ for $p\in\kk\!\Langle y_1,\dots,y_t\Rangle$, and let $(F_\ell)_\ell$ be the nc germ corresponding to $p(f^1,\dots,f^t)$. By \eqref{e:diff}, nc differential operators are linear and satisfy the Leibniz rule,
\begin{equation}\label{e:lei}
\Delta_Y^\ell(h_1h_2)=\sum_{i=0}^\ell \Delta_Y^i h_1 \Delta_Y^{\ell-i} h_2.
\end{equation}
Since $F$ is an an polynomial in $f^1,\dots,f^t$ and $f^i_\ell|_{(\ker\pi)^\ell}=0$ for all $\ell>0$, \eqref{e:lei} implies $F_\ell|_{(\ker\pi)^\ell}=0$ for all $\ell>0$. Moreover, $F_0=p(a^1,\dots, a^t)=0$. Therefore $(F_\ell)_\ell$ is a propagation of $(0)$ as in Theorem \ref{t:min}. On the other hand, $(0)_\ell$ is another such propagation, so $F_\ell=0$ for all $\ell\in\N$ by uniqueness. Therefore $p(f^1,\dots,f^t)=0$.
\end{proof}

\begin{rem}
If $Y$ is semisimple and not similar to a direct sum of scalar points, then we can choose a nonzero nilpotent matrix $a\in\cS(Y)$, so by Theorem \ref{t:alg} there exists a nontrivial nilpotent uniformly analytic function on $\B_{\ve}(Y)$. Note however that $\ve$ is small enough so that $\B_{\ve}(Y)\cap (\C^g \cdot I)=\emptyset$.
\end{rem}

The following restatement of Theorem \ref{t:alg} gives us some further information about the structure of the germ algebras.

\begin{cor}\label{c:hom}
Let $Y$ be a semisimple point. The one-term propagation of Theorem \ref{t:min} gives an homomorphism $\cS(Y)\to\Oua_Y$ that is a left inverse of the evaluation at $Y$.

In particular, if $Y\in\matcn^g$ is irreducible, then $\cO_Y\cong \opm_n(e\cO_Ye)$ and analogously for (uniformly) analytic germs, where $e\in\Oua_Y$ is the one-term propagation of $e_{11}\in\matcn$.
\end{cor}

If $Y$ is irreducible, then it follows by Proposition \ref{p:prime} that the ``corner'' algebra $e\cO_Ye$ in Corollary \ref{c:hom} is prime. A finer structure of this algebra is not known to the authors. Also note that if $Y$ is semisimple, non-irreducible and $\cS(Y)\cong \bigoplus_i \matc{n_i}$, then one cannot conclude that $\cO_Y$ is isomorphic to $\bigoplus_i \opm_{n_i}(\cA_i)$ for some algebras $\cA_i$. This is further supported by Corollary \ref{c:inj} below. Namely, we construct examples of nc functions demonstrating that the canonical homomorphism $\Oua_{Y'\oplus Y''}\to \Oua_{Y'}\times \Oua_{Y''}$ from Remark \ref{r:ss} is not injective.

\begin{lem}\label{l:inj}
Let $Y'\in\matc{s'}^g$ and $Y''\in\matc{s''}^g$ be separated semisimple points, and $Y=Y'\oplus Y''$.
Let $f_1$ be a $Y$-admissible linear map such that
$$(\matc{s'}\oplus\matc{s''})^g\subset \ker f_1.$$
Then there exists an nc function $f$, uniformly analytic on an nc ball about $Y$, such that $\Delta_Y^1f=f_1$ and $f$ vanishes on
$$\bigsqcup_n \big(\matcn\otimes (\matc{s'}\oplus\matc{s''})\big)^g.$$
\end{lem}

\begin{proof}
Since $Y'$ and $Y''$ are separated semisimple points, we have $\cC(Y)\subseteq \matc{s'}\oplus\matc{s''}$. Therefore $(\matc{s'}\oplus\matc{s''})^g$ is a $\cC(Y)$-bimodule, so we can choose the projection $\pi$ from the beginning of the section in such a way that
$$\pi(\matc{s'}\oplus\matc{s''})^g\subseteq (\matc{s'}\oplus\matc{s''})^g.$$
Since $(0,f_1)$ satisfies $\LAC_1(Y)$, there exists $(f_\ell)_\ell$ as in Theorem \ref{t:min}, so there is a uniformly analytic nc function on an nc ball about $Y$ given by
$$f(X)=\sum_{\ell=0}^\infty f_\ell(X-\oplus^nY,\dots,X-\oplus^nY).$$
By Theorem \ref{t:old4} it suffices to show that
\begin{equation}\label{e:72}
f_\ell(Z'^1\oplus Z''^1,\dots,Z'^\ell\oplus Z''^\ell)=0
\end{equation}
for all $\ell\ge1$ and $Z'^j\in\matc{s'}^g$, $Z''^j\in\matc{s''}^g$. For $\ell=1$, \eqref{e:72} holds by the assumption. Since $\cC(Y)\subseteq \matc{s'}\oplus\matc{s''}$, we have
$$\phi\big[\matc{s'}\oplus\matc{s''},Y\big]\subseteq \matc{s'}\oplus\matc{s''}$$
by \eqref{e:0}, where $\phi$ is a right inverse of $S\mapsto [S,Y]$. Therefore
\begin{equation}\label{e:73}
(\phi\circ \pi)(\matc{s'}\oplus\matc{s''})^g\subseteq \matc{s'}\oplus\matc{s''}
\end{equation}
by the choice of $\pi$. Moreover,
\begin{equation}\label{e:74}
\sigma(\matc{s'}\oplus\matc{s''})^g\subseteq (\matc{s'}\oplus\matc{s''})^g,
\end{equation}
where $\sigma:\matcs^g\to\ker\pi$ is the projection onto $\ker\pi$ along $[\matcs,Y]$. Now \eqref{e:72} follows by induction on $\ell$ using the recursive relations \eqref{e:ugly}, \eqref{e:73} and \eqref{e:74}.
\end{proof}

\begin{cor}\label{c:inj}
If $Y'$ and $Y''$ are separated semisimple points, then the canonical homomorphism
$\Oua_{Y'\oplus Y''}\to \Oua_{Y'}\times \Oua_{Y''}$ is not injective.
\end{cor}

\begin{proof}
Let $Y'\in\matc{s'}^g$ and $Y''\in\matc{s''}^g$. Since they are separated semisimple points, we have $\cC(Y'\oplus Y'')\subseteq \matc{s'}\oplus\matc{s''}$. Therefore
$(\matc{s'}\oplus\matc{s''})^g$ is a $\cC(Y)$-bimodule. Furthermore,
$$\dim [\matc{s'+s''},Y'\oplus Y'']+\dim (\matc{s'}\oplus\matc{s''})^g<(s'+s'')^2+g((s')^2+(s'')^2)$$
and so
$$\matc{s'+s''}^g\Big/ \big( [\matc{s'+s''},Y'\oplus Y'']+(\matc{s'}\oplus\matc{s''})^g\big)$$
is a nonzero $\cC(Y)$-bimodule. Therefore there exists a nonzero $\cC(Y)$-bimodule homomorphism $f_1:\matc{s'+s''}^g\to\matc{s'+s''}$ such that
$$[\matc{s'+s''},Y'\oplus Y'']+(\matc{s'}\oplus\matc{s''})^g\subseteq\ker f_1$$
by Lemma \ref{l:n0}. Hence the assumptions on Lemma \ref{l:inj} are satisfied, and let $f$ be the resulting nc function. Then $f\in\ker(\Oua_{Y'\oplus Y''}\to \Oua_{Y'}\times \Oua_{Y''})$.
\end{proof}



\begin{thebibliography}{KK}
	
\bibitem[AK-V15]{AKV}
G. Abduvalieva, D. S. Kaliuzhnyi-Verbovetskyi:
{\it Implicit/inverse function theorems for free noncommutative functions},
J. Funct. Anal. 269 (2015) 2813--2844. 

\bibitem[AM15]{AM0}
J. Agler, J. E. McCarthy:
{\it Global holomorphic functions in several noncommuting variables},
Canad. J. Math. 67 (2015) 241--285. 

\bibitem[AM16]{AM1}
J. Agler, J. E. McCarthy:
{\it Aspects of non-commutative function theory},
Concr. Oper. 3 (2016) 15--24.

\bibitem[AM16']{AM2}
J. Agler, J. E. McCarthy:
{\it The implicit function theorem and free algebraic sets},
Trans. Amer. Math. Soc. 368 (2016) 3157--3175. 

\bibitem[Ami66]{Ami}
S. A. Amitsur:
{\it Rational identities and applications to algebra and geometry},
J. Algebra 3 (1966) 304--359.

\bibitem[Art68]{Art}
M. Artin:
{\it On the solutions of analytic equations},
Invent. Math. 5 (1968) 277--291.

\bibitem[Aug19]{Aug}
M. Augat:
{\it The free Grothendieck theorem},
Proc. London Math. Soc. 118 (2019) 787--825.

\bibitem[BGM05]{BGM}
J. A. Ball, G. Groenewald, T. Malakorn:
{\it Structured noncommutative multidimensional linear systems},
SIAM J. Control Optim. 44 (2005) 1474--1528.

\bibitem[BMV18]{BMV}
J. A. Ball, G. Marx, V. Vinnikov:
{\it Interpolation and transfer-function realization for the noncommutative Schur--Agler class},
Operator theory in different settings and related applications, 23--116,
Oper. Theory Adv. Appl. 262, Birkh{\"a}user/Springer, Cham, 2018.

\bibitem[Ber70]{Ber}
G. M. Bergman:
{\it Skew fields of noncommutative rational functions, after Amitsur},
in S\'eminaire M.P. Sch\"utzenberger, A. Lentin et M. Nivat 1969/70, Probl. Math. Theor. Automates 16, Paris, 1970.

\bibitem[Bla06]{Bla}
B. Blackadar:
{\it Operator algebras. Theory of $C^*$-algebras and von Neumann algebras},
Encyclopaedia of Mathematical Sciences 122, Operator Algebras and Non-commutative Geometry III, Springer-Verlag, Berlin, 2006.

\bibitem[BS99]{BS}
M. Bre\v{s}ar, P. \v{S}emrl:
{\it On locally linearly dependent operators and derivations},
Trans. Amer. Math. Soc. 351 (1999) 1257--1275. 

\bibitem[Coh06]{Coh}
P.M. Cohn:
{\it Free ideal rings and localization in general rings},
New Mathematical Monographs 3, 
Cambridge University Press, Cambridge, 2006.

\bibitem[CR94]{CR}
P.M. Cohn, C. Reutenauer:
{\it A normal form in free fields},
Canad. J. Math. 46 (1994) 517--531.

\bibitem[DK+]{DK}
K. R. Davidson, M. Kennedy:
{\it Noncommutative Choquet theory},
preprint \texttt{arXiv:1905.08436}.

\bibitem[Dic82]{Dic}
W. Dicks:
{\it A commutator test for two elements to generate the free algebra of rank two},
Bull. London Math. Soc. 14 (1982) 48--51.

\bibitem[GMS18]{GMS}
E. Griesenauer, P. S. Muhly, B. Solel:
{\it Boundaries, bundles and trace algebras},
New York J. Math. 24a (2018) 136--154.

\bibitem[HKM11]{HKM}
J. W. Helton, I. Klep, S. McCullough:
{\it Proper analytic free maps},
J. Funct. Anal. 260 (2011) 1476--1490.

\bibitem[HKM12]{HKM1}
J. W. Helton, I. Klep, S. McCullough:
{\it Free analysis, convexity and LMI domains},
Mathematical methods in systems, optimization, and control, 195--219,
Oper. Theory Adv. Appl. 222, Birkh{\"a}user/Springer Basel AG, Basel, 2012. 

\bibitem[HMS18]{HMS}
J. W. Helton, T. Mai, R. Speicher:
{\it Applications of realizations (aka linearizations) to free probability},
J. Funct. Anal. 274 (2018) 1--79.

\bibitem[HMV06]{HMV}
J. W. Helton, S. McCullough, V. Vinnikov:
{\it Noncommutative convexity arises from linear matrix inequalities},
J. Funct. Anal. 240 (2006) 105--191.

\bibitem[K-VV12]{KVV2}
D. S. Kaliuzhnyi-Verbovetskyi, V. Vinnikov:
{\it Noncommutative rational functions, their difference-differential calculus and realizations}, 
Multidimens. Syst. Signal Process. 23 (2012) 49--77.

\bibitem[K-VV14]{KVV3}
D. S. Kalyuzhnyi-Verbovetskyi, V. Vinnikov:
{\it Foundations of free noncommutative function theory},
Mathematical Surveys and Monographs 199,
American Mathematical Society, Providence RI, 2014.

\bibitem[K\v{S}17]{KS}
I. Klep, \v{S}. \v{S}penko:
{\it Free function theory through matrix invariants},
Canad. J. Math. 69 (2017) 408--433.

\bibitem[KVV+]{KVV}
I. Klep, V. Vinnikov, J. Vol\v{c}i\v{c}:
{\it Multipartite rational functions},
preprint \texttt{arXiv:1509.03316}.

\bibitem[KV17]{KV}
I. Klep, J. Vol\v{c}i\v{c}:
{\it Free loci of matrix pencils and domains of noncommutative rational functions},
Comment. Math. Helv. 92 (2017) 105--130.

\bibitem[Kra92]{Kra}
S. G. Krantz:
{\it Function theory of several complex variables},
2nd edition, The Wadsworth \& Brooks/Cole Mathematics Series, Wadsworth \& Brooks/Cole Advanced Books \& Software, Pacific Grove, CA, 1992.

\bibitem[Lam91]{Lam}
T. Y. Lam:
{\it A first course in noncommutative rings},
Graduate Texts in Mathematics 131, Springer-Verlag, New York, 1991.

\bibitem[Lum91]{Lum}
D. Luminet:
{\it Functions of several matrices},
Boll. Un. Mat. Ital. B (7) 11 (1997) 563--586.

\bibitem[Mat89]{Mat}
H. Matsumura:
{\it Commutative ring theory},
2nd edition, Cambridge Studies in Advanced Mathematics 8, Cambridge University Press, Cambridge, 1989.

\bibitem[MS11]{MS}
P. S. Muhly, B. Solel:
{\it Progress in noncommutative function theory},
Sci. China Math. 54 (2011) 2275--2294.

\bibitem[Pas14]{Pas}
J. E. Pascoe:
{\it The inverse function theorem and the Jacobian conjecture for free analysis},
Math. Z. 278 (2014) 987--994.

\bibitem[Pau02]{Pau}
V. Paulsen:
{\it Completely bounded maps and operator algebras},
Cambridge Studies in Advanced Mathematics 78, Cambridge University Press, Cambridge, 2002.

\bibitem[Pis03]{Pis}
G. Pisier:
{\it Introduction to operator space theory},
London Mathematical Society Lecture Note Series 294, Cambridge University Press, Cambridge, 2003.

\bibitem[Pop02]{Pop1}
G. Popescu:
{\it Central intertwining lifting, suboptimization, and interpolation in several variables},
J. Funct. Anal. 189 (2002) 132--154. 

\bibitem[Pop06]{Pop2}
G. Popescu:
{\it Free holomorphic functions on the unit ball of {$B(\mathscr{H})^n$}},
J. Funct. Anal. 241 (2006) 268--333.

\bibitem[Pop08]{Pop3}
G. Popescu:
{\it Free holomorphic functions and interpolation},
Math. Ann. 342 (2008) 1--30. 

\bibitem[Pro76]{Pro}
C. Procesi:
{\it The invariant theory of $n\times n$ matrices},
Adv. Math. 19 (1976) 306--381.

\bibitem[Pro07]{Pro1}
C. Procesi:
{\it Lie groups. An approach through invariants and representations},
Universitext, Springer, New York, 2007.

\bibitem[Rei93]{Rei}
Z. Reichstein:
{\it On automorphisms of matrix invariants},
Trans. Amer. Math. Soc. 340 (1993) 353--371. 

\bibitem[RLL00]{RLL}
M. R{\o}rdam, F. Larsen, N. Laustsen:
{\it An introduction to $K$-theory for $C^*$-algebras},
London Mathematical Society Student Texts 49, Cambridge University Press, Cambridge, 2000.

\bibitem[Row80]{Row}
L. H. Rowen:
{\it Polynomial identities in ring theory},
Pure and Applied Mathematics 84, Academic Press, Inc., New York-London, 1980.

\bibitem[Rya02]{Rya}
R. A. Ryan:
{\it Introduction to Tensor Products of Banach Spaces},
Springer Monographs in Mathematics, Springer-Verlag London, Ltd., London, 2002.

\bibitem[SSS18]{SSS}
G. Salomon, O. Shalit, E. Shamovich:
{\it Algebras of bounded noncommutative analytic functions on subvarieties of the noncommutative unit ball},
Trans. Amer. Math. Soc. 370 (2018) 8639--8690.

\bibitem[Sal99]{Sal}
D. J. Saltman:
{\it Lectures on division algebras},
CBMS Regional Conference Series in Mathematics 94,
American Mathematical Society, Providence RI, 1999.

\bibitem[Shi19]{Shi}
Y. Shitov:
{\it An improved bound for the length of matrix algebras},
Algebra Number Theory 6 (2019) 1501--1507.

\bibitem[Sch85]{Sch}
A. H. Schofield:
{\it Representation of rings over skew fields},
London Mathematical Society Lecture Note Series 92, Cambridge University Press, Cambridge, 1985. 

\bibitem[Tak67]{Tak}
M. Takesaki:
{\it A duality in the representation theory of $C^*$-algebras},
Ann. Math. 85 (1967) 370--382.

\bibitem[Tay72]{Tay1}
J. L. Taylor:
{\it A general framework for a multi-operator functional calculus},
Adv. Math. 9 (1972) 183--252.

\bibitem[Tay73]{Tay2}
J. L. Taylor:
{\it Functions of several noncommuting variables},
Bull. Amer. Math. Soc. 79 (1973) 1--34. 

\bibitem[Voi10]{Voi}
D.-V. Voiculescu:
{\it Free analysis questions II: the Grassmannian completion and the series expansions at the origin}, 
J. Reine Angew. Math. 645 (2010) 155--236. 

\bibitem[Vol18]{Vol}
J. Vol\v{c}i\v{c}:
{\it Matrix coefficient realization theory of noncommutative rational functions}, 
J. Algebra 499 (2018) 397--437.

\end{thebibliography}
\end{document}